\documentclass[12pt,reqno]{amsart}
\usepackage{}
\usepackage{amsmath}
\usepackage{amsfonts}
\usepackage{amssymb}
\usepackage[all,cmtip]{xy}           
\usepackage{xspace}
\usepackage{bbding}
\usepackage{txfonts}
\usepackage[shortlabels]{enumitem}
\usepackage{ifpdf}
\ifpdf
\usepackage[colorlinks,final,hyperindex]{hyperref}
\else
\usepackage[colorlinks,final,hyperindex,hypertex]{hyperref}
\fi
\usepackage{tikz}
\usepackage[active]{srcltx}

\usepackage{tikz}

\makeatletter

\newcommand{\nc}{\newcommand}
\newcommand{\delete}[1]{}

\nc{\mlabel}[1]{\label{#1}}  
\nc{\mcite}[1]{\cite{#1}}  
\nc{\mref}[1]{\ref{#1}}  
\nc{\meqref}[1]{\ref{#1}} 
\nc{\mbibitem}[1]{\bibitem{#1}} 

\delete{
\nc{\mlabel}[1]{\label{#1}  
{\hfill \hspace{1cm}{\bf{{\ }\hfill(#1)}}}}
\nc{\mcite}[1]{\cite{#1}{{\bf{{\ }(#1)}}}}  
\nc{\mref}[1]{\ref{#1}{{\bf{{\ }(#1)}}}}  
\nc{\meqref}[1]{\eqref{#1}{{\bf{{\ }(#1)}}}} 
\nc{\mbibitem}[1]{\bibitem[\bf #1]{#1}} 
}

\newtheorem{theorem}{Theorem}[section]
\newtheorem{prop}[theorem]{Proposition}
\newtheorem{defn}[theorem]{Definition}
\newtheorem{lemma}[theorem]{Lemma}
\newtheorem{coro}[theorem]{Corollary}
\newtheorem{prop-def}[theorem]{Proposition-Definition}

\newtheorem{remark}[theorem]{Remark}

\nc{\tred}[1]{\textcolor{red}{#1}}
\nc{\tblue}[1]{\textcolor{blue}{#1}}
\nc{\tgreen}[1]{\textcolor{green}{#1}}
\nc{\tpurple}[1]{\textcolor{purple}{#1}}
\nc{\btred}[1]{\textcolor{red}{\bf #1}}
\nc{\btblue}[1]{\textcolor{blue}{\bf #1}}
\nc{\btgreen}[1]{\textcolor{green}{\bf #1}}
\nc{\btpurple}[1]{\textcolor{purple}{\bf #1}}

\nc{\li}[1]{\textcolor{red}{#1}}
\nc{\cm}[1]{\textcolor{blue}{Chengming: #1}}
\nc{\xiang}[1]{\textcolor{green}{Xiang: #1}}

\nc{\mgraph}{\mathrm{graph}}
\nc{\qeq}{the $Q$-equations\xspace}
\nc{\adec}{\check{;}} \nc{\aop}{\alpha}
\nc{\dftimes}{\widetilde{\otimes}} \nc{\dfl}{\succ} \nc{\dfr}{\prec}
\nc{\dfc}{\circ} \nc{\dfb}{\bullet} \nc{\dft}{\star}
\nc{\dfcf}{{\mathbf k}} \nc{\apr}{\ast} \nc{\spr}{\cdot}
\nc{\twopr}{\circ} \nc{\sempr}{\ast}
\nc{\disp}[1]{\displaystyle{#1}}
\nc{\bin}[2]{ (_{\stackrel{\scs{#1}}{\scs{#2}}})}  
\nc{\binc}[2]{ \left (\!\! \begin{array}{c} \scs{#1}\\
    \scs{#2} \end{array}\!\! \right )}  
\nc{\bincc}[2]{  \left ( {\scs{#1} \atop
    \vspace{-.5cm}\scs{#2}} \right )}  
\nc{\sarray}[2]{\begin{array}{c}#1 \vspace{.1cm}\\ \hline
    \vspace{-.35cm} \\ #2 \end{array}}
\nc{\bs}{\bar{S}} \nc{\dcup}{\stackrel{\bullet}{\cup}}
\nc{\dbigcup}{\stackrel{\bullet}{\bigcup}} \nc{\etree}{\big |}
\nc{\la}{\longrightarrow} \nc{\fe}{\'{e}} \nc{\rar}{\rightarrow}
\nc{\dar}{\downarrow} \nc{\dap}[1]{\downarrow
\rlap{$\scriptstyle{#1}$}} \nc{\uap}[1]{\uparrow
\rlap{$\scriptstyle{#1}$}} \nc{\defeq}{\stackrel{\rm def}{=}}
\nc{\dis}[1]{\displaystyle{#1}} \nc{\dotcup}{\,
\displaystyle{\bigcup^\bullet}\ } \nc{\sdotcup}{\tiny{
\displaystyle{\bigcup^\bullet}\ }} \nc{\hcm}{\ \hat{,}\ }
\nc{\hcirc}{\hat{\circ}} \nc{\hts}{\hat{\shpr}}
\nc{\lts}{\stackrel{\leftarrow}{\shpr}}
\nc{\rts}{\stackrel{\rightarrow}{\shpr}} \nc{\lleft}{[}
\nc{\lright}{]} \nc{\uni}[1]{\tilde{#1}} \nc{\wor}[1]{\check{#1}}
\nc{\free}[1]{\bar{#1}} \nc{\den}[1]{\check{#1}} \nc{\lrpa}{\wr}
\nc{\curlyl}{\left \{ \begin{array}{c} {} \\ {} \end{array}
    \right .  \!\!\!\!\!\!\!}
\nc{\curlyr}{ \!\!\!\!\!\!\!
    \left . \begin{array}{c} {} \\ {} \end{array}
    \right \} }
\nc{\leaf}{\ell}       
\nc{\longmid}{\left | \begin{array}{c} {} \\ {} \end{array}
    \right . \!\!\!\!\!\!\!}
\nc{\ot}{\otimes} \nc{\sot}{{\scriptstyle{\ot}}}
\nc{\otm}{\overline{\ot}} \nc{\ora}[1]{\stackrel{#1}{\rar}}
\nc{\ola}[1]{\stackrel{#1}{\la}}
\nc{\pltree}{\calt^\pl} \nc{\epltree}{\calt^{\pl,\NC}}
\nc{\rbpltree}{\calt^r} \nc{\scs}[1]{\scriptstyle{#1}}
\nc{\mrm}[1]{{\rm #1}}
\nc{\dirlim}{\displaystyle{\lim_{\longrightarrow}}\,}
\nc{\invlim}{\displaystyle{\lim_{\longleftarrow}}\,}
\nc{\mvp}{\vspace{0.5cm}} \nc{\svp}{\vspace{2cm}}
\nc{\vp}{\vspace{8cm}} \nc{\proofbegin}{\noindent{\bf Proof: }}
\nc{\proofend}{$\blacksquare$ \vspace{0.5cm}}
\nc{\freerbpl}{{F^{\mathrm RBPL}}}
\nc{\sha}{{\mbox{\cyr X}}}  
\nc{\ncsha}{{\mbox{\cyr X}^{\mathrm NC}}} \nc{\ncshao}{{\mbox{\cyr
X}^{\mathrm NC,\,0}}}
\nc{\shpr}{\diamond}    
\nc{\shprm}{\overline{\diamond}}    
\nc{\shpro}{\diamond^0}    
\nc{\shprr}{\diamond^r}     
\nc{\shpra}{\overline{\diamond}^r} \nc{\shpru}{\check{\diamond}}
\nc{\catpr}{\diamond_l} \nc{\rcatpr}{\diamond_r}
\nc{\lapr}{\diamond_a} \nc{\sqcupm}{\ot} \nc{\lepr}{\diamond_e}
\nc{\vep}{\varepsilon} \nc{\labs}{\mid\!} \nc{\rabs}{\!\mid}
\nc{\hsha}{\widehat{\sha}} \nc{\lsha}{\stackrel{\leftarrow}{\sha}}
\nc{\rsha}{\stackrel{\rightarrow}{\sha}} \nc{\lc}{\lfloor}
\nc{\rc}{\rfloor} \nc{\tpr}{\sqcup} \nc{\nctpr}{\vee}
\nc{\plpr}{\star} \nc{\rbplpr}{\bar{\plpr}} \nc{\sqmon}[1]{\langle
#1\rangle} \nc{\forest}{\calf} \nc{\ass}[1]{\alpha({#1})}
\nc{\altx}{\Lambda_X} \nc{\vecT}{\vec{T}} \nc{\onetree}{\bullet}
\nc{\Ao}{\check{A}} \nc{\seta}{\underline{\Ao}}
\nc{\deltaa}{\overline{\delta}} \nc{\trho}{\tilde{\rho}}

\nc{\mmbox}[1]{\mbox{\ #1\ }} \nc{\ann}{\mrm{ann}}
\nc{\Aut}{\mrm{Aut}} \nc{\can}{\mrm{can}} \nc{\twoalg}{{two-sided
algebra}\xspace} \nc{\colim}{\mrm{colim}} \nc{\Cont}{\mrm{Cont}}
\nc{\rchar}{\mrm{char}} \nc{\cok}{\mrm{coker}} \nc{\dtf}{{R-{\rm
tf}}} \nc{\dtor}{{R-{\rm tor}}}

\nc{\depth}{{\mrm d}} \nc{\Div}{{\mrm Div}} \nc{\End}{\mrm{End}}
\nc{\Ext}{\mrm{Ext}} \nc{\Fil}{\mrm{Fil}} \nc{\Frob}{\mrm{Frob}}
\nc{\Gal}{\mrm{Gal}} \nc{\GL}{\mrm{GL}} \nc{\Hom}{\mrm{Hom}}
\nc{\hsr}{\mrm{H}} \nc{\hpol}{\mrm{HP}} \nc{\id}{\mrm{id}}
\nc{\im}{\mrm{im}} \nc{\incl}{\mrm{incl}} \nc{\length}{\mrm{length}}
\nc{\LR}{\mrm{LR}} \nc{\mchar}{\rm char} \nc{\NC}{\mrm{NC}}
\nc{\mpart}{\mrm{part}} \nc{\pl}{\mrm{PL}} \nc{\ql}{{\QQ_\ell}}
\nc{\qp}{{\QQ_p}} \nc{\rank}{\mrm{rank}} \nc{\rba}{\rm{RBA }}
\nc{\rbas}{\rm{RBAs }} \nc{\rbpl}{\mrm{RBPL}} \nc{\rbw}{\rm{RBW }}
\nc{\rbws}{\rm{RBWs }} \nc{\rcot}{\mrm{cot}}
\nc{\rest}{\rm{controlled}\xspace} \nc{\rdef}{\mrm{def}}
\nc{\rdiv}{{\rm div}} \nc{\rtf}{{\rm tf}} \nc{\rtor}{{\rm tor}}
\nc{\res}{\mrm{res}} \nc{\SL}{\mrm{SL}} \nc{\Spec}{\mrm{Spec}}
\nc{\tor}{\mrm{tor}} \nc{\Tr}{\mrm{Tr}} \nc{\mtr}{\mrm{sk}}

\nc{\ab}{\mathbf{Ab}} \nc{\Alg}{\mathbf{Alg}}
\nc{\Algo}{\mathbf{Alg}^0} \nc{\Bax}{\mathbf{Bax}}
\nc{\Baxo}{\mathbf{Bax}^0} \nc{\RB}{\mathbf{RB}}
\nc{\RBo}{\mathbf{RB}^0} \nc{\BRB}{\mathbf{RB}}
\nc{\Dend}{\mathbf{DD}} \nc{\bfk}{{\bf k}} \nc{\bfone}{{\bf 1}}
\nc{\base}[1]{{a_{#1}}} \nc{\detail}{\marginpar{\bf More detail}
    \noindent{\bf Need more detail!}
    \svp}
\nc{\Diff}{\mathbf{Diff}} \nc{\gap}{\marginpar{\bf
Incomplete}\noindent{\bf Incomplete!!}
    \svp}
\nc{\FMod}{\mathbf{FMod}} \nc{\mset}{\mathbf{MSet}}
\nc{\rb}{\mathrm{RB}} \nc{\Int}{\mathbf{Int}}
\nc{\Mon}{\mathbf{Mon}}
\nc{\remarks}{\noindent{\bf Remarks: }}
\nc{\OS}{\mathbf{OS}} 
\nc{\Rep}{\mathbf{Rep}} \nc{\Rings}{\mathbf{Rings}}
\nc{\Sets}{\mathbf{Sets}} \nc{\DT}{\mathbf{DT}}

\nc{\BA}{{\mathbb A}} \nc{\CC}{{\mathbb C}} \nc{\DD}{{\mathbb D}}
\nc{\EE}{{\mathbb E}} \nc{\FF}{{\mathbb F}} \nc{\GG}{{\mathbb G}}
\nc{\HH}{{\mathbb H}} \nc{\LL}{{\mathbb L}} \nc{\NN}{{\mathbb N}}
\nc{\QQ}{{\mathbb Q}} \nc{\RR}{{\mathbb R}} \nc{\TT}{{\mathbb T}}
\nc{\VV}{{\mathbb V}} \nc{\ZZ}{{\mathbb Z}}

\nc{\calao}{{\mathcal A}} \nc{\cala}{{\mathcal A}}
\nc{\calc}{{\mathcal C}} \nc{\cald}{{\mathcal D}}
\nc{\cale}{{\mathcal E}} \nc{\calf}{{\mathcal F}}
\nc{\calfr}{{{\mathcal F}^{\,r}}} \nc{\calfo}{{\mathcal F}^0}
\nc{\calfro}{{\mathcal F}^{\,r,0}} \nc{\oF}{\overline{F}}
\nc{\calg}{{\mathcal G}} \nc{\calh}{{\mathcal H}}
\nc{\cali}{{\mathcal I}} \nc{\calj}{{\mathcal J}}
\nc{\call}{{\mathcal L}} \nc{\calm}{{\mathcal M}}
\nc{\caln}{{\mathcal N}} \nc{\calo}{{\mathcal O}}
\nc{\calp}{{\mathcal P}} \nc{\calr}{{\mathcal R}}
\nc{\calt}{{\mathcal T}} \nc{\caltr}{{\mathcal T}^{\,r}}
\nc{\calu}{{\mathcal U}} \nc{\calv}{{\mathcal V}}
\nc{\calw}{{\mathcal W}} \nc{\calx}{{\mathcal X}}
\nc{\CA}{\mathcal{A}}

\nc{\fraka}{{\mathfrak a}} \nc{\frakB}{{\mathfrak B}}
\nc{\frakb}{{\mathfrak b}} \nc{\frakd}{{\mathfrak d}}
\nc{\oD}{\overline{D}} \nc{\frakF}{{\mathfrak F}}
\nc{\frakg}{{\mathfrak g}} \nc{\frakm}{{\mathfrak m}}
\nc{\frakM}{{\mathfrak M}} \nc{\frakMo}{{\mathfrak M}^0}
\nc{\frakp}{{\mathfrak p}} \nc{\frakS}{{\mathfrak S}}
\nc{\frakSo}{{\mathfrak S}^0} \nc{\fraks}{{\mathfrak s}}
\nc{\os}{\overline{\fraks}} \nc{\frakT}{{\mathfrak T}}
\nc{\oT}{\overline{T}}
\nc{\frakX}{{\mathfrak X}} \nc{\frakXo}{{\mathfrak X}^0}
\nc{\frakx}{{\mathbf x}}
\nc{\frakTx}{\frakT}      
\nc{\frakTa}{\frakT^a}        
\nc{\frakTxo}{\frakTx^0}   
\nc{\caltao}{\calt^{a,0}}   
\nc{\ox}{\overline{\frakx}} \nc{\fraky}{{\mathfrak y}}
\nc{\frakz}{{\mathfrak z}} \nc{\oX}{\overline{X}}

\font\cyr=wncyr10

\nc{\redtext}[1]{\textcolor{red}{#1}}

\begin{document}

\title{On quadri-bialgebras}

\author{Xiang Ni}

\address{Department of Mathematics, Caltech, Pasadena, CA91125,
USA}\email{nixiang85@gmail.com}

\author{Chengming Bai}

\address{Chern Institute of Mathematics \& LPMC, Nankai University,
Tianjin 300071, P.R. China} \email{baicm@nankai.edu.cn}

\begin{abstract}
We introduce the notion of a quadri-bialgebra, which gives a
bialgebra theory for the quadri-algebra introduced by Aguiar and
Loday. We show that a quadri-bialgebra is equivalent to a Manin
triple of dendriform algebras associated to a nondegenerate
2-cocycle, and to a Manin triple of quadri-algebras associated to
a nondegenerate invariant bilinear form. Quadri-bialgebras also
come from a variation of the classical Yang-Baxter equation,
called \qeq. Moreover, quadri-bialgebras fit into the framework of
construction of Rota-Baxter operators and Nijenhuis operators on
the double spaces of quadri-algebras.
\end{abstract}

\subjclass[2010]{16W30, 17A30, 18D50 }

\keywords{Quadri-algebra, bialgebra, dendriform algebra, classical
Yang-Baxter equation, Rota-Baxter operator, Nijenhuis operator}

\maketitle

\tableofcontents

\allowdisplaybreaks

\section{Introduction}

Motivated by the study of periodicity in algebraic K-theory, J.-L. Loday introduced the notion of
a dendriform algebra~\cite{Lo1}. It has attracted quite much interest because of its connections with various fields in
mathematics and physics (see \cite{EMP} and the references therein).

There is a remarkable fact that a Rota-Baxter operator (of weight
zero), which first arose in probability
 theory~\cite{Bax} and later became popular in combinatorics and other areas~\cite{Gub,R1,R2}, on an associative algebra naturally
gives a  dendriform algebra structure on the underlying vector space of the associative algebra~\cite{Ag1,Ag2,E1}.
Such unexpected relationship
 between dendriform algebras in the field of operads and algebraic topology and
 Rota-Baxter operators in the field of combinatorics and probability theory was extended to several other structures.
In order to determine the algebraic structures
 behind
a pair of commuting Rota-Baxter operators (on an associative
algebra), which appeared, for example, in the space of the linear
endomorphisms of an infinitesimal
 bialgebra, Aguiar and Loday introduced the notion of
a quadri-algebra~\cite{AL}, which is a vector space equipped with four binary operations
 satisfying nine axioms. Moreover, quadri-algebras have deep
relationships with combinatorics and the theory of Hopf algebras~\cite{AL,EG1}. A quadri-algebra is also regarded as the underlying
algebra structure of a dendriform algebra with a nondegenerate
2-cocycle~\cite{Bai3} and is one of the first examples of operad splitting~\cite{BBGN,EG1,Lo2}.

This paper establishes a bialgebra theory for quadri-algebras,
utilizing methods from the theory of Lie bialgebras~\cite{CP,D}.
Explicitly, in the finite-dimensional case,  we consider an
analogue of a Manin triple of Lie algebras which is equivalent to
a Lie bialgebra, namely, a Manin triple of dendriform algebras
associated to a nondegenerate 2-cocycle. We find that it is in
fact equivalent to certain bialgebra structure of the underlying
quadri-algebra, which leads to the notion of quadri-bialgebra.
Furthermore, it is interesting to find that such a structure is
also equivalent to another Manin triple on the level of
quadri-algebras, that is, a Manin triple of quadri-algebras with a
nondegenerate invariant bilinear form. Quadri-bialgebras have
certain similar properties of Lie bialgebras. For example, there
are the so-called coboundary quadri-bialgebras which lead to a
construction from an analogue of the classical Yang-Baxter
equation and there also exists a ``Drinfeld   double" construction
for a finite-dimensional quadri-bialgebra. We would like to point
out that the finite-dimensional restriction is imposed to
motivated the definition of a quadri-bialgebra in its equivalence
to Manin triples and matched pairs. The notion of a
quadri-bialgebra itself, once established (in
Definition~\mref{de:quadbial}), can be applied with no dimensional
restrictions.

Moreover, we find that quadri-bialgebras fit into the framework of
constructions of certain linear operators, such as Rota-Baxter
operators and Nijenhuis operators
  in combinatorics~\cite{Bax, E1,E2, R1,R2}, renormalization of
perturbative quantum field theory (pQFT)~\cite{CK, EG2, EGK1,EGK2}
and quantum physics~\cite{CGM}, on the double spaces.

The paper is organized as follows. In Section~\mref{sec:dend}, we recall some notions and
basic facts on dendriform algebras and quadri-algebras. In Section~\mref{sec:manin}, we introduce the notion of a Manin triple of dendriform algebras
associated to a nondegenerate 2-cocycle and then interpret it in
terms of matched pairs of dendriform algebras. In Section~\mref{sec:match}, we
consider the notion of a Manin triple of quadri-algebras associated
to a nondegenerate invariant bilinear form and then interpret it in
terms of matched pairs of quadri-algebras. We also show the
equivalence between the two Manin triples introduced in the two sections. In Section~\mref{sec:qbial}, we
define a quadri-bialgebra as a bialgebra
structure equivalent to the aforementioned  Manin triples by assuming
that there is a quadri-algebra structure on the dual space. In
Section~\mref{sec:cob}, we study the coboundary case which leads to a
construction from certain algebraic equations, which could be regarded as
analogues of the classical Yang-Baxter equation. In Section~\mref{sec:qeq}, we
recall some results in \cite{Bai3} which reduce the aforementioned
algebraic equations to a simple form, namely, \qeq referring
to a set of two equations. We list some properties of \qeq
including the ones given in \cite{Bai3} from another point of view.
 In Section~\mref{sec:double}, we construct families of Nijenhuis operators and
Rota-Baxter operators on certain double spaces of quadri-algebras,
including the Drinfeld $Q$-doubles obtained from
quadri-bialgebras.

\section{Notations and backgrounds}
\mlabel{sec:dend}

Throughout this paper, all algebras and vector spaces are finite-dimensional over a
fixed base field $\mathbb F$. We put together notations and conventions to be used throughout the paper.

(a) Let $V$ be a vector space.  Let $\mathfrak{B}:V\otimes
V\rightarrow\mathbb F$ be a symmetric or skew-symmetric bilinear
form on a vector space $V$. A subspace $W$ is called {\bf isotropic}
if $W\subset W^{\perp}$, where
\begin{equation}
W^{\perp}=\{x\in V|\mathfrak{B}(x,y)=0,\forall y\in
W\}.
\mlabel{eq:1.1}
\end{equation}
A subspace $W$ is called {\bf Lagrangian} if $W=W^{\perp}$.

 (b) Let $(A,\diamond)$ be a vector space with a
bilinear operation $\diamond: A\otimes A\rightarrow A$. Let
$L_\diamond(x)$ and $R_\diamond(x)$ denote the left and right
multiplication operators respectively, that is,
$L_\diamond(x)y=R_\diamond(y)x=x\diamond y$ for all $x,y\in A$. We
also simply denote them by $L(x)$ and $R(x)$ respectively if the
meaning of $\diamond$ is clear from the context. Moreover, let
$L_\diamond, R_\diamond:A\rightarrow \frak{gl}(A)$ be the linear
maps defined by $x\mapsto L_\diamond(x)$ and $x\mapsto
R_\diamond(x)$ respectively.

(c) Let $V$ be a vector space and let $r=\sum_i{a_i\otimes b_i}\in
V\otimes V$. Set
\begin{equation}
r_{12}:=\sum_ia_i\otimes b_i\otimes 1,\quad r_{13}:=\sum_{i}a_i\otimes
1\otimes b_i,\quad
r_{23}:=\sum_i1\otimes a_i\otimes b_i,
\mlabel{eq:1.2}
\end{equation}
where $1$ is a symbol playing a similar role of unit. If in
addition, there exists a bilinear operation $\diamond: V\otimes
V\rightarrow V$ on $V$, then the operation between two $r$s is in an
obvious way. For example,
\begin{equation}
r_{12}\diamond r_{13}=\sum_{i,j}a_i\diamond a_j\otimes b_i\otimes
b_j,\; r_{13}\diamond r_{23}=\sum_{i,j}a_i\otimes a_j\otimes
b_i\diamond b_j,\;r_{23}\diamond r_{12}=\sum_{i,j}a_j\otimes
a_i\diamond  b_j\otimes b_i.
\mlabel{eq:1.3}
\end{equation}
Note that Eq.~(\meqref{eq:1.3}) is independent of the existence of
the unit.

(d) Let $V$ be a vector space. Let $\tau:V\otimes V\to V\otimes V$
be the flip map defined as
    \begin{equation}
    \tau(u\otimes v)=v\otimes u, \quad \forall u,v \in V.
    \end{equation}

 (e) Let $V$ be a vector space.  Any $r\in V\otimes V$ is
identified with the linear map $T_r: V^*\rightarrow V$ defined by
\begin{equation}
\langle u^*\otimes v^*,\ r\rangle=\langle u^*,T_r(v^*) \rangle,\quad
\forall u^*, v^*\in V^*,
\mlabel{eq:1.4}
\end{equation}
where $\langle, \rangle$ is the canonical paring between $V$ and
$V^*$. $r\in V\otimes V$ is called {\bf nondegenerate} if the above
induced linear map $T_r$ is invertible.

(f) Let $V_1,V_2$ be two vector spaces and let $T:V_1\rightarrow
V_2$ be a linear map. Denote the dual (linear) map by
$T^*:V_2^*\rightarrow V_1^*$ defined by
\begin{equation}
\langle  v_1,T^*(v_2^*)\rangle   =\langle  T(v_1),v_2^*\rangle
,\;\;\forall v_1\in V_1, v_2^*\in V_2^*.
\mlabel{eq:1.5}
\end{equation}

(g) Let $A$ be a vector space with a set of bilinear operations
and let $V$ be a vector space. For any linear map $\rho:
A\rightarrow \frak{gl}(V)$, define a linear map $\rho^* : A
\rightarrow \frak{gl}(V^*)$ by
\begin{equation}
\langle \rho^*(x)v^*, u\rangle = \langle v^*, \rho(x)u\rangle,\
\forall x \in A, u\in V, v^*\in V^*.
\mlabel{eq:1.6}
\end{equation}
Note that in this case, $\rho^*$ is different from the one given
by Eq.~(\meqref{eq:1.5}) which regards $\frak{gl}(V)$ as a vector
space also.

(h) Let $V$ be a vector space, we sometimes use 1 to denote the
identity transformation on $V$.

Here are the notions that motivated our study.

\begin{defn}
{\rm \cite{Lo1} A {\bf dendriform algebra} $(A,\prec,\succ)$ is a
vector space $A$ together with two bilinear operations
$\prec,\succ: A\otimes A\rightarrow A$ such that (for all
$x,y,z\in A$)
\begin{equation}(x\prec y)\prec
z=x\prec(y\star z), (x\succ y)\prec z=x\succ(y\prec z),(x\star
y)\succ z=x\succ(y\succ z),
\mlabel{eq:2.1}
\end{equation}
where $x\star y=x\prec y+x\succ y$. Moreover, a  {\bf homomorphism}
between two dendriform algebras
 is defined to be a linear map (between the two underlying vector spaces) which preserves the operations respectively. }
\mlabel{de:2.1}
\end{defn}

\begin{prop}{\rm \cite{Lo1}} Let $(A,\prec,\succ)$ be a dendriform
algebra. Then the bilinear operation given by
\begin{equation}
x\star y:=x\prec y+x\succ y,\;\forall x,y\in A,
\mlabel{eq:2.2}
\end{equation}
defines an associative algebra, which is called the {\bf
associated associative algebra} and denoted by $(As(A),\star)$.
\mlabel{rk:2.2}
\end{prop}

\begin{defn}{\rm \cite{Ag3,Bai2}
Let $(A,\prec,\succ)$ be a dendriform algebra and let $V$ be a
vector space. Let
$l_{\prec},l_{\succ},r_{\prec},r_{\succ}:A\rightarrow
\frak{gl}(V)$ be four linear maps. Then $V$ or
$(V,l_{\prec},r_{\prec},l_{\succ},r_{\succ})$ is called a {\bf
bimodule} of $A$ if  the following equations hold (for all $x,y\in
A$):}
\begin{equation}r_{\prec}(y)r_{\prec}(x)=r_{\prec}(x\star y),
r_{\prec}(y)l_{\prec}(x)=l_{\prec}(x)r_{\star}(y), l_{\prec}(x\prec
y)=l_{\prec}(x)l_{\star}(y),
\mlabel{eq:2.3}
\end{equation}
\begin{equation}
r_{\prec}(y)r_{\succ}(x)=r_{\succ}(x\prec y),
r_{\prec}(y)l_{\succ}(x)=l_{\succ}(x)r_{\prec}(y), l_{\prec}(x\succ
y)=l_{\succ}(x)l_{\prec}(y),
\mlabel{eq:2.4}
\end{equation}
\begin{equation}r_{\succ}(y)r_{\star}(x)=r_{\succ}(x\succ
y),   r_{\succ}(y)l_{\star}(x)=l_{\succ}(x)r_{\succ}(y),
l_{\succ}(x\star y)=l_{\succ}(x)l_{\succ}(y).
\mlabel{eq:2.5}
\end{equation}
\mlabel{de:2.3}
\end{defn}

\begin{prop}{\rm \cite{Bai2}} Let
$(V,l_{\prec},r_{\prec},l_{\succ},r_{\succ})$ be a bimodule of a
dendriform algebra $(A,\prec$, $\succ)$. Then
$(V^*,-r_\succ^*,l_\succ^*+l_\prec^*,r_\succ^*+r_\prec^*,
-l_\prec^*)$ is a bimodule of $(A,\prec,\succ)$.
\mlabel{pp:2.4}
\end{prop}

\begin{defn}
{\rm \cite{Bai2,K} Let
$(V,l_{\prec},r_{\prec},l_{\succ},r_{\succ})$ be a bimodule of a
dendriform algebra $(A,\prec$, $\succ)$. A linear map
$T:V\rightarrow A$ is called an {\bf $\mathcal O$-operator} associated
to $(V,l_{\prec},r_{\prec},l_{\succ},r_{\succ})$ if $T$ satisfies
\small{
\begin{equation}
T(u)\prec T(v)=T(l_\prec (T(u))v+r_\prec(T(v))u), T(u)\succ
T(u)=T(l_\succ (T(u))v+r_\succ(T(v))u),\forall u,v\in
V.
\mlabel{eq:2.6}
\end{equation}}
}
\mlabel{de:2.5}
\end{defn}

\begin{defn}\cite{AL}
A {\bf quadri-algebra} $(A,\nwarrow,\nearrow,\swarrow,\searrow)$
is a vector space $A$ together with four bilinear operations
$\nwarrow,\nearrow,\swarrow$ and $\searrow:A\otimes A\rightarrow
A$ satisfying the axioms below (for all $x,y,z\in A$)
\begin{equation}(x\nwarrow y)\nwarrow z=x\nwarrow(y\star z),   (x\nearrow
y)\nwarrow z=x\nearrow(y\prec z),   (x\wedge y)\nearrow
z=x\nearrow(y\succ z),
\mlabel{eq:2.7}
\end{equation}
\begin{equation} (x\swarrow
y)\nwarrow z=x\swarrow(y\wedge z),   (x\searrow y)\nwarrow
z=x\searrow(y\nwarrow  z),   (x\vee y)\nearrow z=x\searrow
(y\nearrow z),
\mlabel{eq:2.8}
\end{equation}
\begin{equation}(x\prec y)\swarrow
z=x\swarrow(y\vee z),   (x\succ y)\swarrow z=x\searrow (y\swarrow
z),   (x\star y)\searrow z=x\searrow(y\searrow z),
\mlabel{eq:2.9}
\end{equation}
where
\begin{equation} x\succ y:=x\nearrow y+x\searrow y, x\prec
y:=x\nwarrow y+x\swarrow y,
\mlabel{eq:2.10}
\end{equation}
\begin{equation}
 x\vee y:=x\swarrow y+x\searrow y,   x\wedge y:=x\nwarrow
y+x\nearrow y,
\mlabel{eq:2.11}
\end{equation}
\begin{equation}x\star y:=x\searrow
y+x\nearrow y+x\nwarrow y+x\swarrow y=x\succ y+x\prec y=x\vee
y+x\wedge y.
\mlabel{eq:2.12}
\end{equation}
A {\bf homomorphism} between two quadri-algebras
 is defined as a linear map (between the two underlying vector spaces) which preserves the
 operations respectively.
\mlabel{de:2.6}
\end{defn}

\begin{prop}{\rm \cite{AL}} Let $(A,\nwarrow,\nearrow,\swarrow,\searrow)$
be a quadri-algebra.

\begin{enumerate}
\item The bilinear operation given by Eq.~(\ref{eq:2.10}) defines
a dendriform algebra $A_h:=(A,\prec,\succ)$, which is called the
{\bf associated horizontal dendriform algebra}. \mlabel{it:2.7a}
\item The bilinear operation given by Eq.~(\ref{eq:2.11}) defines
a dendriform algebra $A_v:=(A,\wedge,\vee)$, which is called the
{\bf associated vertical dendriform algebra}. \mlabel{it:2.7b}
\item The bilinear operation given by Eq.~(\ref{eq:2.12}) defines
an associative algebra $(A,\star)$, which is called the {\bf
associated associative algebra}, which is the associated
associative algebra of both the dendriform algebras
$(A,\prec,\succ)$ and $(A,\wedge,\vee)$. \mlabel{it:2.7c}
\end{enumerate}
\mlabel{pp:2.7}
\end{prop}

\begin{prop}{\rm \cite{Bai3}} Let $A$ be a vector space
with four bilinear operations denoted by $\nwarrow,\nearrow$,
$\swarrow$ and $\searrow:A\otimes A\rightarrow A$. Then the
following conditions are equivalent:

\begin{enumerate}
\item  $(A,\nwarrow,\nearrow,\swarrow,\searrow)$ is a
quadri-algebra; \mlabel{it:2.8a} \item $(A,\prec,\succ)$ defined
by Eq.~(\ref{eq:2.10}) is a dendriform algebra and $(A,
L_{\swarrow},R_{\nwarrow}, L_{\searrow}, R_{\nearrow})$ is a
bimodule. \mlabel{it:2.8b} \item $(A,\wedge,\vee)$ defined by Eq.~
(\ref{eq:2.11}) is a dendriform algebra and $(A,
L_{\nearrow},R_{\nwarrow}, L_{\searrow},R_{\swarrow})$ is a
bimodule. \mlabel{it:2.8c}
\end{enumerate}
\mlabel{pp:2.8}
\end{prop}

\begin{coro}{\rm \cite{Bai3}} Let  $(A,\nwarrow,\nearrow,\swarrow,\searrow)$ be a
quadri-algebra. Then $(A^*,-R_{\nearrow}^*,L_{\vee}^*,R_{\wedge}^*$,
$-L_{\swarrow}^*)$ is a bimodule of the associated horizontal
dendriform algebra $(A,\prec,\succ)$ and $(A^*,-R_{\swarrow}^*$,
$L_{\succ}^*$, $ R_{\prec}^*$, $-L_{\nearrow}^*)$ is a bimodule of
the associated vertical dendriform algebra $(A,\wedge,\vee)$.
\mlabel{co:2.9}
\end{coro}

In the following, we will focus on the study of the
associated vertical dendriform algebras of quadri-algebras in this
paper. The corresponding study on the associated horizontal
dendriform algebras is completely similar.

\begin{defn}{\rm \cite{Bai2} Let $(A,\wedge,\vee)$ be a dendriform
algebra and let $(As(A), \star)$  be the associated  associative
algebra. Suppose that $\mathfrak{B}:A\otimes A\rightarrow\mathbb
F$ is a symmetric bilinear form. $\mathfrak{B}$ is called a {\bf
$2$-cocycle} on $A$ if $\mathfrak{B}$ satisfies
\begin{equation}
\mathfrak{B}(x\star y,z)=\mathfrak{B}(y,z\wedge
x)+\mathfrak{B}(x,y\vee z),    \forall x,y,z\in
A.
\mlabel{eq:2.13}
\end{equation}}
\mlabel{de:2.10}
\end{defn}

\begin{prop}{\rm \cite{Bai3}}
 Let $(A,\wedge,\vee)$ be a dendriform algebra
equipped with a nondegenerate $2$-cocycle $\mathfrak{B}$. Define
four bilinear operations $\nwarrow,
\nearrow,\swarrow,\searrow:A\otimes A\rightarrow A$ by
\begin{equation}\mathfrak{B}(x\nwarrow y,z)=\mathfrak{B}(x,y\star z),
\mathfrak{B}(x\nearrow y,z)=-\mathfrak{B}(y,z\vee x),
\mlabel{eq:2.14}
\end{equation}
\begin{equation}\mathfrak{B}(x\swarrow y,z)=-\mathfrak{B}(x,y\wedge z),
\mathfrak{B}(x\searrow y,z)=\mathfrak{B}(y,z\star x),\;\forall
x,y,z\in A.
\mlabel{eq:2.15}
\end{equation}
Then $(A,\nwarrow,\nearrow,\swarrow,\searrow)$ is a quadri-algebra such
that $(A,\wedge,\vee)$ is the associated vertical dendriform
algebra.
\mlabel{pp:2.11}
\end{prop}

\section{Manin triples of
dendriform algebras associated to a nondegenerate 2-cocycle and matched pairs of dendriform algebras}
\mlabel{sec:manin}

This section studies Manin triples of dendriform algebras and the
closely related matched pairs of dendriform algebras.

\begin{defn}  A {\bf Manin triple of
dendriform algebras associated to a nondegenerate 2-cocycle} is a
triple of dendriform algebras $(A,A^{+},A^{-})$ together with a
nondegenerate $2$-cocycle $\mathfrak{B}$ on $A$, such that:
\begin{enumerate}
\item $A^{+}$ and $A^{-}$ are subalgebras of $A$; \mlabel{it:3.1a}
\item $A=A^{+}\oplus A^{-}$ as vector spaces;
\mlabel{it:3.1b}\item $A^{+}$ and $A^{-}$ are isotropic with
respect to $\mathfrak{B}$. \mlabel{it:3.1c}
\end{enumerate}
It is denoted by $(A,A^{+},A^{-},\mathfrak{B})$. A {\bf
homomorphism} between two Manin triples $(A,A^{+},A^{-},\mathfrak{B}_A)$ and
$(B,B^{+},B^{-},\mathfrak{B}_B)$ of dendriform algebras
associated to a nondegenerate 2-cocycle
is a homomorphism of dendriform
algebras $\varphi:A\rightarrow B$ such that
\begin{equation}\varphi(A^{+})\subset B^{+},\;
\varphi(A^{-})\subset B^{-},\;
\mathfrak{B}_A(x,y)=\mathfrak{B}_B(\varphi(x),
\varphi(y)),\; \forall x,y\in A.
\mlabel{eq:3.1}
\end{equation}
\mlabel{de:3.1}
\end{defn}

\begin{defn}
Let $(A,\wedge,\vee)$ be a dendriform algebra. Suppose that there
is a dendriform algebra structure on the dual space $A^*$. If
there is a dendriform algebra structure on the direct sum of the
underlying vector spaces of $A$ and $A^*$ such that $A$ and $A^*$
are subalgebras and the natural symmetric bilinear form on
$A\oplus A^*$ given by
\begin{equation}\mathfrak{B}_S(x+a^*,y+b^*):=\langle  a^*,y\rangle +\langle
x,b^*\rangle ,\; \forall x,y\in A; a^*,b^*\in A^*,
\mlabel{eq:3.2}
\end{equation} is
a $2$-cocycle, then $(A\oplus A^*,A,A^*,\mathfrak{B}_S)$ is called a
{\bf standard Manin triple of dendriform algebras associated to
$\mathfrak{B}_S$.}
\mlabel{de:3.2}
\end{defn}

 Obviously, a standard Manin triple of dendriform
algebras is a Manin triple of dendriform algebras. Conversely,
we have

\begin{prop}
Every Manin triple of dendriform algebras associated to a
nondegenerate 2-cocycle is isomorphic to a standard one.
\mlabel{pp:3.3}
\end{prop}

\begin{proof} Since in this case $A^{-}$ and $(A^+)^*$ are identified by
the nondegenerate $2$-cocycle, the dendriform algebra structure on
$A^{-}$ is transferred to  $(A^+)^*$. Hence the dendriform algebra
structure on $A^{+}\oplus A^{-}$ is transferred to
$A^{+}\oplus(A^{+})^*$. Then the conclusion follows.
\end{proof}

\begin{prop} {\rm \cite{Bai2}} Let $(A,\wedge_A,\vee_A)$ and $(B,\wedge_B,\vee_B)$ be two
dendriform algebras. Suppose that there are linear maps
$l_{\wedge_A},r_{\wedge_A},l_{\vee_A},r_{\vee_A}:A\rightarrow \frak{gl}(B)$
and $l_{\wedge_B},r_{\wedge_B},l_{\vee_B},r_{\vee_B}:B\rightarrow
\frak{gl}(A)$ such that $(l_{\wedge_A},r_{\wedge_A},l_{\vee_A},r_{\vee_A})$
is a bimodule of $A$ and
$(l_{\wedge_B},r_{\wedge_B},l_{\vee_B},r_{\vee_B})$ is a bimodule of
$B$, and they satisfy the following conditions:
\begin{equation}(l_{\wedge_B}(a)x)\wedge_Ay+l_{\wedge_B}(r_{\wedge_A}(x)a)y=
l_{\wedge_B}(a)(x\star_Ay) ,
\mlabel{eq:3.3}
\end{equation}
\begin{equation}
l_{\wedge_B}(l_{\wedge_A}(x)a)y+(r_{\wedge_B}(a)x)\wedge_Ay=
x\wedge_A(l_{\star_B}(a)y)+r_{\wedge_B}(r_{\star_A}(y)a)x,
\mlabel{eq:3.4}
\end{equation}
\begin{equation}r_{\wedge_B}(a)(x\wedge_Ay)=r_{\wedge_B}(l_{\star_A}(y)a)x+
x\wedge_A(r_{\star_B}(a)y),
\mlabel{eq:3.5}
\end{equation}
\begin{equation}
(l_{\vee_B}(a)x)\wedge_Ay+l_{\wedge_B}(r_{\vee_A}(x)a)y=
l_{\vee_B}(a)(x\wedge_Ay),
\mlabel{eq:3.6}
\end{equation}\begin{equation}l_{\wedge_B}(l_{\vee_A}(x)a)y+(r_{\vee_B}(a)x)\wedge_Ay=
x\vee_A(l_{\wedge_B}(a)y)+r_{\vee_B}(r_{\wedge_A}(y)a)x,
\mlabel{eq:3.7}
\end{equation}\begin{equation}
r_{\wedge_B}(a)(x\vee_Ay)=r_{\vee_B}(l_{\wedge_A}(y)a)x+
x\vee_A(r_{\wedge_B}(a)y),
\mlabel{eq:3.8}
\end{equation}\begin{equation}(l_{\star_B}(a)x)\vee_Ay+l_{\vee_B}
(r_{\star_A}(x)a)y=l_{\vee_B}(a)(x\vee_Ay),
\mlabel{eq:3.9}
\end{equation}\begin{equation}
l_{\vee_B}(l_{\star_A}(x)a)y+(r_{\star_B}(a)x)\vee_Ay=x\vee_A(l_{\vee_B}(a)y)
+r_{\vee_B}(r_{\vee_A}(y)a)x,
\mlabel{eq:3.10}
\end{equation}
\begin{equation}r_{\vee_B}(a)(x\star_Ay)=r_{\vee_B}(l_{\vee_A}(y)a)x+x\vee_A(r_{\vee_B}(a)y),
\mlabel{eq:3.11}
\end{equation}\begin{equation}
(l_{\wedge_A}(x)a)\wedge_Bb+l_{\wedge_A}(r_{\wedge_B}(a)x)b=
l_{\wedge_A}(x)(a\star_Bb),
\mlabel{eq:3.12}
\end{equation}\begin{equation}l_{\wedge_A}(l_{\wedge_B}(a)x)b+(r_{\wedge_A}(x)a)\wedge_Bb
= a\wedge_B(l_{\star_A}(x)b)
+r_{\wedge_A}(r_{\star_B}(b)x)a,
\mlabel{eq:3.13}
\end{equation}\begin{equation}
r_{\wedge_A}(x)(a\wedge_Bb)=r_{\wedge_A}(l_{\star_B}(b)x)a+
a\wedge_B(r_{\star_A}(x)b),
\mlabel{eq:3.14}
\end{equation}\begin{equation}(l_{\vee_A}(x)a)\wedge_Bb+l_{\wedge_A}(r_{\vee_B}(a)x)b=
l_{\vee_A}(x)(a\wedge_Bb),
\mlabel{eq:3.15}
\end{equation}\begin{equation}
l_{\wedge_A}(l_{\vee_B}(a)x)b+(r_{\vee_A}(x)a)\wedge_Bb=
a\vee_B(l_{\wedge_A}(x)b)
+r_{\vee_A}(r_{\wedge_B}(b)x)a,
\mlabel{eq:3.16}
\end{equation}\begin{equation}r_{\wedge_A}(x)(a\vee_Bb)=r_{\vee_A}(l_{\wedge_B}(b)x)a+
a\vee_B(r_{\wedge_A}(x)b),
\mlabel{eq:3.17}
\end{equation}\begin{equation}
(l_{\star_A}(x)a)\vee_Bb+l_{\vee_A}(r_{\star_B}(a)x)b=
l_{\vee_A}(x)(a\vee_Bb),
\mlabel{eq:3.18}
\end{equation}\begin{equation}l_{\vee_A}(l_{\star_B}(a)x)b+(r_{\star_A}(x)a)\vee_Bb=a\vee_B(l_{\vee_A}(x)b)
+r_{\vee_A}(r_{\vee_B}(b)x)a,
\mlabel{eq:3.19}
\end{equation}\begin{equation}
r_{\vee_A}(x)(a\star_Bb)=r_{\vee_A}(l_{\vee_B}(b)x)a+
a\vee_B(r_{\vee_A}(x)b),
\mlabel{eq:3.20}
\end{equation} for all $x,y\in A, a,b\in B$.
Then there is a dendriform algebra structure on the vector space
$A\oplus B$ given by
\begin{equation}(x+a)\wedge(y+b):=x\wedge_Ay+l_{\wedge_B}(a)y+r_{\wedge_B}(b)x+a\wedge_Bb+
l_{\wedge_A}(x)b+r_{\wedge_A}(y)a,
\mlabel{eq:3.21}
\end{equation}\begin{equation}
(x+a)\vee(y+b):=x\vee_Ay+l_{\vee_B}(a)y+r_{\vee_B}(b)x+a\vee_Bb+
l_{\vee_A}(x)b+r_{\vee_A}(y)a,
\mlabel{eq:3.22}
\end{equation}
for all $x,y\in A,a,b\in B.$ This dendriform algebra, given by the
data $(A,B$, $l_{\wedge_A},r_{\wedge_A}$, $l_{\vee_A},r_{\vee_A},
l_{\wedge_B}$, $r_{\wedge_B}$, $l_{\vee_B}$, $r_{\vee_B})$ which
is called a {\bf matched pair of dendriform algebras}, is denoted
by $A\bowtie_{l_{\wedge_A},r_{\wedge_A},l_{\vee_A},
r_{\vee_A}}^{l_{\wedge_B},r_{\wedge_B},l_{\vee_B},r_{\vee_B}}B$ or
simply $A\bowtie B$. On the other hand, every dendriform algebra
which is the direct sum of the underlying vector spaces of two
dendriform subalgebras can be obtained from the above way.
\mlabel{pp:3.4}
\end{prop}

\begin{prop} Let $(A,\nwarrow_A,\nearrow_A, \swarrow_A,\searrow_A)$ be a
quadri-algebra. Suppose that there is a quadri-algebra structure
$\nwarrow_{A^*},\nearrow_{A^*}, \swarrow_{A^*},\searrow_{A^*}$ on
the dual space $A^*$. Let $A_v$ and $(A^*)_v$ be the associated
vertical dendriform algebras respectively. Then there exists a
dendriform algebra structure on the vector space $A\oplus A^*$
such that $A_v$ and $(A^*)_v$ are isotropic subalgebras associated
to the $2$-cocycle $\frak B_S$ given by Eq.~(\ref{eq:3.2}), that
is, $(A_v\oplus (A^*)_v, A_v,(A^*)_v,\mathfrak{B}_S)$  is a
standard Manin triple of dendriform algebras associated to the
2-cocycle $\frak B_S$, if and only if
$(A_v,(A^*)_v,-R_{\swarrow_A}^*,L_{\succ_A}^*,R_{\prec_A}^*,
-L_{\nearrow_A}^*,-R_{\swarrow_{A^*}}^*$, $L_{\succ_{A^*}}^*$,
$R_{\prec_{A^*}}^*$, $-L_{\nearrow_{A^*}}^*)$ is a matched pair of
dendriform algebras. \mlabel{pp:3.5}
\end{prop}

\begin{proof}
If $(A_v,(A^*)_v,-R_{\swarrow_A}^*,L_{\succ_A}^*,R_{\prec_A}^*,
-L_{\nearrow_A}^*,-R_{\swarrow_{A^*}}^*,L_{\succ_{A^*}}^*,R_{\prec_{A^*}}^*,
-L_{\nearrow_{A^*}}^*)$ is a matched pair of dendriform algebras,
then it is straightforward to check that the  bilinear form $\frak
B_S$ given by Eq.~(\meqref{eq:3.2}) on
$A_v\bowtie_{-R_{\swarrow_A}^*,L_{\succ_A}^*,R_{\prec_A}^*,
-L_{\nearrow_A}^*}^{-R_{\swarrow_{A^*}}^*,L_{\succ_{A^*}}^*,R_{\prec_{A^*}}^*,
-L_{\nearrow_{A^*}}^*}(A^*)_v$ is a 2-cocycle. So $(A_v\oplus
(A_v)^*,A_v,(A^*)_v,\mathfrak{B}_S)$ is a standard Manin triple of
dendriform algebras associated to $\frak B_S$.

Conversely, if $(A_v\oplus (A_v)^*,A_v,(A_v)^*,\mathfrak{B}_S)$ is
a standard Manin triple of dendriform algebras associated to the
2-cocycle $\frak B_S$, then for all $x,y\in A,a^*,b^*\in A^*$, we
have
$$\langle x\wedge a^*,y\rangle =\mathfrak{B}_S(y,x\wedge
a^*)=-\mathfrak{B}_S(y\swarrow_Ax,a^*)=\langle
y,-R_{\swarrow_A}^*(x)a^*\rangle ,$$
$$\langle x\star a^*,b^*\rangle =\mathfrak{B}_S(b^*,x\star
a^*)=\mathfrak{B}_S(a^*\searrow_{A^*}b^*,x)= \langle
L_{\searrow_{A^*}}^*(a^*)x,b^*\rangle, $$  $$\langle x\vee
a^*,b^*\rangle =\mathfrak{B}_S(b^*,x\vee
a^*)=-\mathfrak{B}_S(a^*\nearrow_{A^*}b^*,x)= \langle
-L_{\nearrow_{A^*}}^*(a^*)x,b^*\rangle.$$ So $x\wedge
a^*=-R_{\swarrow_A}^*(x)a^*+L_{\succ_{A^*}}^*(a^*)x$. Similarly,
we show that
$$x\vee a^*=R_{\prec_A}^*(x)a^*-L_{\nearrow_{A^*}}^*(a^*)x,
a^*\wedge x=-R_{\swarrow_{A^*}}^*(a^*)x+L_{\succ_A}^*(x)a^*,
a^*\vee x=R_{\prec_{A^*}}^*(a^*)x-L_{\nearrow_A}^*(x)a^*,$$ for
all $x\in A,a^*\in A^*$. Hence
$(A_v,(A^*)_v,-R_{\swarrow_A}^*,L_{\succ_A}^*,R_{\prec_A}^*,
-L_{\nearrow_A}^*,-R_{\swarrow_{A^*}}^*,L_{\succ_{A^*}}^*,R_{\prec_{A^*}}^*,
-L_{\nearrow_{A^*}}^*)$ is a matched pair of dendriform
algebras.\end{proof}

\section{Bimodules and matched pairs of quadri-algebras}
\mlabel{sec:match}

In this section we introduce the notion of a Manin triple of
quadri-algebras associated to a nondegenerate invariant bilinear
form and establish its equivalence with a matched pair of
quadri-algebras on the one hand and with a Manin triple of
dendriform algebras associated to a nondegenerate 2-cocycle on the
other hand.

\begin{defn}{\rm \cite{Bai3} Let $(A,\nwarrow,\nearrow,\swarrow,\searrow)$ be a
quadri-algebra and let $V$ be a vector space. Let
$l_{\circ},r_{\circ}:A\rightarrow \frak{gl}(V)$ be eight linear maps, where
$\circ\in\{\nwarrow,\nearrow,\swarrow,\searrow\}$. Then $V$ or
$(V,l_{\nwarrow},r_{\nwarrow},l_{\nearrow},
r_{\nearrow}$, $l_{\swarrow}$, $r_{\swarrow}$,
$l_{\searrow},r_{\searrow})$) is called a {\bf bimodule} of $A$ if
for all $x,y\in A$,
\begin{equation} r_{\nwarrow}(y)r_{\nwarrow}(x)=r_{\nwarrow}(x\star y),
r_{\nwarrow}(y)l_{\nwarrow}(x)=l_{\nwarrow}(x)r_{\star}(y),
l_{\nwarrow}(x\nwarrow
y)=l_{\nwarrow}(x)l_{\star}(y),
\mlabel{eq:4.1}
\end{equation}\begin{equation}
r_{\nwarrow}(y)r_{\nearrow}(x)=r_{\nearrow}(x\prec y),
r_{\nwarrow}(y)l_{\nearrow}(x)=l_{\nearrow}(x)r_{\prec}(y),
l_{\nwarrow}(x\nearrow
y)=l_{\nearrow}(x)l_{\prec}(y),
\mlabel{eq:4.2}
\end{equation}\begin{equation}r_{\nearrow}(y)r_{\wedge}(x)=r_{\nearrow}(x\succ
y), r_{\nearrow}(y)l_{\wedge}(x)=l_{\nearrow}(x)r_{\succ}(y),
l_{\nearrow}(x\wedge
y)=l_{\nearrow}(x)l_{\succ}(y),
\mlabel{eq:4.3}
\end{equation}\begin{equation}r_{\nwarrow}(y)r_{\swarrow}(x)=r_{\swarrow}(x\wedge
y), r_{\nwarrow}(y)l_{\swarrow}(x)=l_{\swarrow}(x)r_{\wedge}(y),
l_{\nwarrow}(x\swarrow
y)=l_{\swarrow}(x)l_{\wedge}(y),
\mlabel{eq:4.4}
\end{equation}\begin{equation}r_{\nwarrow}(y)r_{\searrow}(x)=r_{\searrow}(x\nwarrow
y), r_{\nwarrow}(y)l_{\searrow}(x)=l_{\searrow}(x)r_{\nwarrow}(y),
l_{\nwarrow}(x\searrow
y)=l_{\searrow}(x)l_{\nwarrow}(y),
\mlabel{eq:4.5}
\end{equation}\begin{equation}r_{\nearrow}(y)r_{\vee}(x)=r_{\searrow}(x\nearrow
y), r_{\nearrow}(y)l_{\vee}(x)=l_{\searrow}(x)r_{\nearrow}(y),
l_{\nearrow}(x\vee
y)=l_{\searrow}(x)l_{\nearrow}(y),
\mlabel{eq:4.6}
\end{equation}\begin{equation}r_{\swarrow}(y)r_{\prec}(x)=r_{\swarrow}(x\vee
y), r_{\swarrow}(y)l_{\prec}(x)=l_{\swarrow}(x)r_{\vee}(y),
l_{\swarrow}(x\prec
y)=l_{\swarrow}(x)l_{\vee}(y),
\mlabel{eq:4.7}
\end{equation}\begin{equation}r_{\swarrow}(y)r_{\succ}(x)=r_{\searrow}(x\swarrow
y), r_{\swarrow}(y)l_{\succ}(x)=l_{\searrow}(x)r_{\swarrow}(y),
l_{\swarrow}(x\succ
y)=l_{\searrow}(x)l_{\swarrow}(y),
\mlabel{eq:4.8}
\end{equation}\begin{equation}r_{\searrow}(y)r_{\star}(x)=r_{\searrow}(x\searrow
y), r_{\searrow}(y)l_{\star}(x)=l_{\searrow}(x)r_{\searrow}(y),
l_{\searrow}(x\star
y)=l_{\searrow}(x)l_{\searrow}(y).
\mlabel{eq:4.9}
\end{equation}}
\mlabel{de:4.1}
\end{defn}

In fact, $(V,l_{\nwarrow},r_{\nwarrow},l_{\nearrow},
r_{\nearrow},l_{\swarrow},r_{\swarrow},l_{\searrow},r_{\searrow})$
is a bimodule of a quadri-algebra
$(A,\nwarrow,\nearrow,\swarrow,\searrow)$ if and only if the
direct sum $A\oplus V$ of the underlying vector spaces of $A$ and
$V$ is turned into a quadri-algebra (the {\bf semidirect product})
by defining multiplications in $A\oplus V$ by (we still denote
them by $\nwarrow,\nearrow,\swarrow,\searrow$):
\begin{equation}(x_1+u_1)\circ(x_2+u_2):=x_1\circ
x_2+ (l_{\circ}(x_1)u_2+r_{\circ}(x_2)u_1), \forall x_1,x_2\in
A,u_1,u_2\in V,
\circ\in\{\nwarrow,\nearrow,\swarrow,\searrow\}.
\mlabel{eq:4.10}
\end{equation} We
denote it by $A\ltimes_{l_{\nwarrow},r_{\nwarrow},l_{\nearrow},
r_{\nearrow},l_{\swarrow},r_{\swarrow},l_{\searrow},r_{\searrow}}V$
or simply $A\ltimes V$.

\begin{lemma}{\rm \cite{Bai3}}
Let $(A,\nwarrow,\nearrow,\swarrow,\searrow)$ be a quadri-algebra.
If $(V,l_{\nwarrow},r_{\nwarrow},l_{\nearrow},
r_{\nearrow},l_{\swarrow},r_{\swarrow},l_{\searrow}$,
$r_{\searrow})$ is a bimodule of $A$, then
$(V^*,r_{\searrow}^*,l_{\star}^*,-r_{\vee}^*,-l_{\prec}^*,
-r_{\succ}^*,-l_{\wedge}^*,r_{\star}^*,l_{\nwarrow}^*)$ is a
bimodule of $A$.
\mlabel{lem:4.2}
\end{lemma}

\begin{defn}{\rm
   Let
$(A,\nwarrow_A,\nearrow_A,\swarrow_A,\searrow_A)$ and
$(B,\nwarrow_B,\nearrow_B,\swarrow_B,\searrow_B)$ be two
quadri-algebras. Suppose that there exist linear maps
$l_{\nwarrow_A},r_{\nwarrow_A},l_{\nearrow_A},r_{\nearrow_A},
l_{\swarrow_A},r_{\swarrow_A},l_{\searrow_A},r_{\searrow_A}:A\rightarrow
\frak{gl}(B)$ and $l_{\nwarrow_B},r_{\nwarrow_B}$,
$l_{\nearrow_B}, r_{\nearrow_B}$, $l_{\swarrow_B},r_{\swarrow_B}$,
$l_{\searrow_B},r_{\searrow_B}:B\rightarrow \frak{gl}(A)$ such
that for all $x,y\in A,a,b\in B$,
\begin{equation}(x+a)\diamond(y+b):=x\diamond_Ay+l_{\diamond_B}(a)y+
r_{\diamond_B}(b)x+a\diamond_Bb+ l_{\diamond_A}(x)b+
r_{\diamond_A}(y)a,\diamond\in\{\nwarrow,\nearrow,\swarrow,\searrow\},
\mlabel{eq:4.11}
\end{equation}
define a quadri-algebra structure on $A\oplus B$. Then
$(A,B,l_{\nwarrow_A},r_{\nwarrow_A}$,
$l_{\nearrow_A},r_{\nearrow_A}$, $ l_{\swarrow_A}$,
$r_{\swarrow_A}, l_{\searrow_A}$, $r_{\searrow_A},l_{\nwarrow_B}$,
$r_{\nwarrow_B},l_{\nearrow_B}$, $r_{\nearrow_B},
l_{\swarrow_B},r_{\swarrow_B},l_{\searrow_B},r_{\searrow_B})$ is
called {\bf a matched pair of quadri-algebras} and the
quadri-algebra structure on $A\oplus B$ is denoted by
$A\bowtie_{l_{\nwarrow_A},r_{\nwarrow_A},l_{\nearrow_A},r_{\nearrow_A},
l_{\swarrow_A},r_{\swarrow_A},l_{\searrow_A},r_{\searrow_A}}^{l_{\nwarrow_B},
r_{\nwarrow_B},l_{\nearrow_B},r_{\nearrow_B},
l_{\swarrow_B},r_{\swarrow_B},l_{\searrow_B},r_{\searrow_B}}B$ or
simply $A\bowtie B$.} \mlabel{de:4.3}
\end{defn}

\begin{remark}  {\rm Similar to Proposition~\mref{pp:3.4}, one can also write down the necessary and
sufficient conditions that the above linear maps make $A\oplus B$ into
a quadri-algebra. We omit the details since we
will not use such a conclusion in this paper. In
particular, in this case, $A$ and $B$ are bimodules of $B$ and $A$
respectively.}
\mlabel{rk:4.4}
\end{remark}

\begin{prop}
 Let
$(A,B,l_{\nwarrow_A},r_{\nwarrow_A},l_{\nearrow_A},r_{\nearrow_A},
l_{\swarrow_A},r_{\swarrow_A},l_{\searrow_A}$,
$r_{\searrow_A},l_{\nwarrow_B}, r_{\nwarrow_B}$, $l_{\nearrow_B}$,
$r_{\nearrow_B}$, $ l_{\swarrow_B}$, $r_{\swarrow_B}$,
$l_{\searrow_B}$, $r_{\searrow_B})$ be a matched pair of
quadri-algebras. Then $((A)_v,(B)_v,
l_{\nwarrow_A}+l_{\nearrow_A},r_{\nwarrow_A}+r_{\nearrow_A},
l_{\swarrow_A}$ $+l_{\searrow_A},r_{\swarrow_A}+r_{\searrow_A},
l_{\nwarrow_B}+l_{\nearrow_B},r_{\nwarrow_B}+r_{\nearrow_B},
l_{\swarrow_B}+l_{\searrow_B},r_{\swarrow_B}+r_{\searrow_B})$ is a
matched pair of dendriform algebras. \mlabel{pp:4.5}
\end{prop}

\begin{proof} It is straightforward.\end{proof}

\begin{prop}  Let $(A,\nwarrow,\nearrow,\swarrow,\searrow)$ be a
quadri-algebra. Suppose that there is a quadri-algebra structure
$(\nwarrow_*,\nearrow_*,\swarrow_*,\searrow_*)$ on the dual space
$A^*$. Then $(A_v,(A^*)_v,-R_{\swarrow}^*,L_{\succ}^*,R_{\prec}^*,
-L_{\nearrow}^*$, $-R_{\swarrow_*}^*,L_{\succ_*}^*, R_{\prec_*}^*,
-L_{\nearrow_*}^*)$ is a matched pair of dendriform algebras if
and only if $(A,A^*,R_{\searrow}^*,L_{\star}^*$, $-R_{\vee}^*,
-L_{\prec}^*,-R_{\succ}^*,-L_{\wedge}^*, R_{\star}^*,
L_{\nwarrow}^*,R_{\searrow_*}^*,L_{\star_*}^*,-R_{\vee_*}^*,
-L_{\prec_*}^*,-R_{\succ_*}^*,-L_{\wedge_*}^*,R_{\star_*}^*,
L_{\nwarrow_*}^*)$ is a matched pair of quadri-algebras.
\mlabel{pp:4.6}
\end{prop}

\begin{proof} By Proposition~\mref{pp:4.5}, we only need to prove the ``only if"
part of the conclusion. In fact, if
$(A_v,(A^*)_v,-R_{\swarrow}^*,L_{\succ}^*,R_{\prec}^*,
-L_{\nearrow}^*,-R_{\swarrow_*}^*,L_{\succ_*}^*,R_{\prec_*}^*,
-L_{\nearrow_*}^*)$ is a matched pair of dendriform algebras, then
by Proposition~\ref{pp:3.5}, we find that
$(A_v\oplus(A^*)_v,A_v,(A^*)_v,\mathfrak{B}_S)$ is a standard
Manin triple of dendriform algebras associated to the 2-cocycle
$\frak B_S$. Hence there exists a quadri-algebra structure on
$A_v\bowtie_{-R_{\swarrow}^*,L_{\succ}^*,R_{\prec}^*,
-L_{\nearrow}^*}^{-R_{\swarrow_*}^*,L_{\succ_*}^*,R_{\prec_*}^*,
-L_{\nearrow_*}^*}(A^*)_v$ given by Proposition~\mref{pp:2.11}
which we denote by $(\nwarrow_{\bullet},
\nearrow_{\bullet},\swarrow_{\bullet},\searrow_{\bullet}$).
Moreover, $A$ and $A^*$ are the quadri-subalgebras. For all
$x,y\in A,a^*,b^*\in A^*$, we have
\begin{eqnarray*}
&&\langle x\nwarrow_{\bullet}a^*,y\rangle
=\mathfrak{B}_S(x,a^*\star_{\bullet}y)= \langle
x,L_{\searrow}^*(y)a^*\rangle =\langle
R_{\searrow}^*(x)a^*,y\rangle,\\
&&\langle x\nwarrow_{\bullet}a^*,b^*\rangle =
\mathfrak{B}_S(x,a^*\star_*b^*)=\langle
L_{\star_*}^*(a^*)x,b^*\rangle. \end{eqnarray*}So
$x\nwarrow_{\bullet}a^*=R_{\searrow}^*(x)a^*+L_{\star_*}^*(a^*)x$.
Similarly, we show that
\begin{eqnarray*}
&&x\nearrow_{\bullet}a^*=-R_{\vee}^*(x)a^*-L_{\prec_*}^*(a^*)x,
x\swarrow_{\bullet}a^*=-R_{\succ}^*(x)a^*-L_{\wedge_*}^*(a^*)x,\\
&&x\searrow_{\bullet}a^*=R_{\star}^*(x)a^*+L_{\nwarrow_*}^*(a^*)x,
a^*\nwarrow_{\bullet}x=R_{\searrow_*}^*(a^*)x+L_{\star}^*(x)a^*,\\
&&a^*\nearrow_{\bullet}x=-R_{\vee_*}^*(a^*)x-L_{\prec}^*(x)a^*,
a^*\swarrow_{\bullet}x=-R_{\succ_*}^*(a^*)x-L_{\wedge}^*(x)a^*,
\end{eqnarray*}
and
$a^*\searrow_{\bullet}x=R_{\star_*}^*(a^*)x+L_{\nwarrow}^*(x)a^*.$
Therefore $(A,A^*,R_{\searrow}^*,L_{\star}^*,-R_{\vee}^*,
-L_{\prec}^*,-R_{\succ}^*,-L_{\wedge}^*,R_{\star}^*,
L_{\nwarrow}^*$, $R_{\searrow_*}^*$, $L_{\star_*}^*,-R_{\vee_*}^*,
-L_{\prec_*}^*, -R_{\succ_*}^*,-L_{\wedge_*}^*,R_{\star_*}^*,
L_{\nwarrow_*}^*)$ is a matched pair of
quadri-algebras.\end{proof}

In fact, the equivalence in Proposition~\mref{pp:4.6} between two matched pairs can be
interpreted in terms of their corresponding Manin triples as follows, in Corollary~\mref{co:4.10}.

\begin{defn}{\rm  Let $(A,\nwarrow,\nearrow,\swarrow,\searrow)$ be a
quadri-algebra and let $\mathfrak{B}$ be a symmetric bilinear
form. If $\mathfrak{B}$ satisfies
Eqs.~(\meqref{eq:2.14})-(\meqref{eq:2.15}), then $\mathfrak{B}$ is
called {\bf invariant} on $A$.} \mlabel{de:4.7}
\end{defn}

\begin{prop}{\rm \cite{Bai3}} Let $(A,\nwarrow,\nearrow,\swarrow,\searrow)$ be a
quadri-algebra and let $\mathfrak{B}$ be a symmetric bilinear
form. If $\mathfrak{B}$ is invariant on $A$, then $\mathfrak{B}$
is a 2-cocycle of the associated vertical dendriform algebra
$(A_v,\wedge,\vee)$. Conversely, if $\mathfrak{B}$ is a
nondegenerate 2-cocycle of a dendriform algebra, then
$\mathfrak{B}$ is invariant on the quadri-algebra given by
Eqs.~(\meqref{eq:2.14})-(\meqref{eq:2.15}). \mlabel{pp:4.8}
\end{prop}

\begin{defn}
{\rm Let $(A,\nwarrow,\nearrow,\swarrow,\searrow)$ be a
quadri-algebra. Suppose that there is a quadri-algebra structure
on $A^*$. If there is a quadri-algebra structure on the direct sum
of the underlying vector space of $A$ and $A^*$ such that $A$ and
$A^*$ are quadri-subalgebras and the bilinear form
$\mathfrak{B}_S$ on $A\oplus A^*$ given by Eq.~(\ref{eq:3.2}) is
invariant, then $(A\bowtie A^*,A,A^*,\mathfrak{B}_S)$ is called a
{\bf (standard) Manin triple of quadri-algebras associated to the
nondegenerate invariant bilinear form $\frak B_S$}.}
\mlabel{de:4.9}
\end{defn}

By Proposition~\ref{pp:4.8}, the following conclusion is obvious:

\begin{coro} The quadruple $(A\bowtie
A^*,A,A^*,\mathfrak{B}_S)$ is a standard Manin triple of
quadri-algebras associated to the nondegenerate invariant bilinear
form $\frak B_S$ if and only if $(A_v\bowtie
(A_v)^*,A_v,A_v^*,{\mathfrak B}_S)$ is a standard Manin triple of
dendriform algebras associated to the nondegenerate 2-cocycle
$\frak B_S$. \mlabel{co:4.10}
\end{coro}

\begin{remark}{\rm By Proposition~\ref{pp:4.6}, it is obvious that a standard Manin triple of quadri-algebras
associated to the nondegenerate invariant bilinear form $\frak
B_S$ can be interpreted in terms of a matched pair of
quadri-algebras (cf. Theorem~\mref{thm:5.3})}. \mlabel{rk:4.11}
\end{remark}

\section{Quadri-bialgebras}
\mlabel{sec:qbial}

With the preparations in the previous sections, we now introduce
the notion of a quadri-bialgebra and establish its equivalence
with a Manin triple of dendriform algebras and of quadri-algebras,
and a matched pair of dendriform algebras and of quadri-algebras.

\begin{prop} Let
$(A,\nwarrow,\nearrow,\swarrow,\searrow,\alpha,\beta,\tilde{\alpha},\tilde{\beta})$
be a quadri-algebra equipped with four cooperations
$\alpha,\beta,\tilde{\alpha},\tilde{\beta}:A\rightarrow A\otimes A$.
Suppose that
$\alpha^*,\beta^*,\tilde{\alpha}^*,\tilde{\beta}^*:A^*\otimes
A^*\subset(A\otimes A)^*\rightarrow A^*$ induce a quadri-algebra
structure on $A^*$. Set
$\nwarrow_*:=\alpha^*,\nearrow_*:=\beta^*,\swarrow_*:=
\tilde{\alpha}^*,\searrow_*:=\tilde{\beta}^*$. Then
$(A_v,(A^*)_v,-R_{\swarrow}^*,L_{\succ}^*,R_{\prec}^*,
-L_{\nearrow}^*,-R_{\swarrow_*}^*,L_{\succ_*}^*, R_{\prec_*}^*,
-L_{\nearrow_*}^*)$ is a matched pair of dendriform algebras if
and only if the following equations hold:
\begin{equation}\tilde{\alpha}(x\star y)= (R_{\wedge}(y)\otimes
1)\tilde{\alpha}(x)+ (1\otimes
L_{\succ}(x))\tilde{\alpha}(y),
\mlabel{eq:5.1}
\end{equation}\begin{equation}
\beta(x\star y)=(R_{\prec}(y)\otimes 1)\beta(x)+(1\otimes
L_{\vee}(x))\beta(y),
\mlabel{eq:5.2}
\end{equation}\begin{equation}(\beta+\tilde{\beta})(x\wedge
y)=(R_{\nwarrow}(y)\otimes 1)(\beta+\tilde{\beta})(x)+(1\otimes
L_{\wedge}(x))\tilde{\beta}(y), \mlabel{eq:5.3}\end{equation}
\begin{equation}(\alpha+\tilde{\alpha})(x\wedge
y)= (R_{\wedge}(y)\otimes 1)(\alpha+\tilde{\alpha})(x)+ (1\otimes
L_{\nearrow}(x))\tilde{\alpha}(y), \mlabel{eq:5.4}
\end{equation}\begin{equation}(\beta+\tilde{\beta})(x\vee y)=
(1\otimes L_{\vee}(x))(\beta+\tilde{\beta})(y)+
(R_{\swarrow}(y)\otimes
1)\beta(x),
\mlabel{eq:5.5}
\end{equation}\begin{equation}(\alpha+\tilde{\alpha})(x\vee
y)= (1\otimes L_{\searrow}(x))(\alpha+\tilde{\alpha})(y)+
(R_{\vee}(y)\otimes
1)\alpha(x),
\mlabel{eq:5.6}
\end{equation}\begin{equation}(\alpha+\beta)(x\succ
y)=(1\otimes L_{\searrow}(x))(\alpha+ \beta)(y)+(R_{\succ}(y)\otimes
1)\alpha(x),
\mlabel{eq:5.7}
\end{equation}\begin{equation}(\tilde{\alpha}+\tilde{\beta})(x\succ
y)= (1\otimes L_{\succ}(x))(\tilde{\alpha}+\tilde{\beta})(y)+
(R_{\nearrow}(y)\otimes
1)\tilde{\alpha}(x),
\mlabel{eq:5.8}
\end{equation}\begin{equation}(\alpha+\beta)(x\prec
y)=(R_{\prec}(y)\otimes 1)(\alpha+ \beta)(x)+(1\otimes
L_{\swarrow}(x))\beta(y),
\mlabel{eq:5.9}
\end{equation}\begin{equation}(\tilde{\alpha}+\tilde{\beta})(x\prec
y)= (R_{\nwarrow}(y)\otimes 1)(\tilde{\alpha}+\tilde{\beta})(x)+
(1\otimes
L_{\prec}(x))\tilde{\beta}(y),
\mlabel{eq:5.10}\end{equation}\begin{equation}(1\otimes
L_{\succ}(y)-R_{\wedge}(y)\otimes 1)\tau\beta(x)= (1\otimes
R_{\prec}(x)-L_{\vee}(x)\otimes 1)\tilde{\alpha}(y),
\mlabel{eq:5.11}
\end{equation}\begin{equation}(1\otimes
R_{\nwarrow}(x)-L_{\vee}(x)\otimes 1)(\alpha+
\tilde{\alpha})(y)=(1\otimes L_{\nearrow}(y))\tau\beta(x)-
(R_{\vee}(y)\otimes
1)\tau\tilde{\beta}(x),
\mlabel{eq:5.12}
\end{equation}\begin{equation}(R_{\wedge}(y)\otimes
1-1\otimes L_{\searrow}(y))(\tau\beta+
\tau\tilde{\beta})(x)=(L_{\wedge}(x)\otimes 1)\alpha(y)-(1\otimes
R_{\swarrow}(x))\tilde{\alpha}(y),
\mlabel{eq:5.13}
\end{equation}\begin{equation}(1\otimes
L_{\succ}(x)-R_{\nwarrow}(x)\otimes 1)(\tau\alpha+
\tau\beta)(y)=(1\otimes R_{\succ}(y))\tilde{\beta}(x)-
(L_{\swarrow}(y)\otimes
1)\tilde{\alpha}(x),
\mlabel{eq:5.14}
\end{equation}\begin{equation}(1\otimes
L_{\searrow}(x)- R_{\prec}(x)\otimes
1)(\tau\tilde{\alpha}+\tau\tilde{\beta})(y)= (1\otimes
R_{\nearrow}(y))\beta(x)- (L_{\prec}(y)\otimes
1)\alpha(x),
\mlabel{eq:5.15}
\end{equation}\begin{equation}(\alpha+\tilde{\alpha}+\beta+\tilde{\beta})(x\swarrow
y)= (R_{\swarrow}(y)\otimes 1)(\alpha+\beta)(x)+(1\otimes
L_{\swarrow}(x))(\beta+\tilde{\beta})(y),
\mlabel{eq:5.16}
\end{equation}\begin{equation}(\alpha+\beta+\tilde{\alpha}+\tilde{\beta})(x\nearrow
y)= (R_{\nearrow}(y)\otimes 1)(\alpha+\tilde{\alpha})(x)+(1\otimes
L_{\nearrow}(x))(\tilde{\alpha}+\tilde{\beta})(y),
\mlabel{eq:5.17}
\end{equation}\begin{equation}(L_{\swarrow}(y)\otimes
1)(\alpha+\tilde{\alpha})(x)+ (1\otimes
L_{\nearrow}(x))(\tau\alpha+\tau\beta)(y)= (R_{\swarrow}(x)\otimes
1)(\tau\tilde{\alpha}+ \tau\tilde{\beta})(y)+(1\otimes
R_{\nearrow}(y))(\beta+ \tilde{\beta})(x), \mlabel{eq:5.18}
\end{equation}
for all $x,y\in A$. \mlabel{pp:5.1}
\end{prop}

\begin{proof} From Proposition~\ref{pp:3.5}, we only need to prove that each of Eqs.~(\meqref{eq:5.1})-(\meqref{eq:5.18}) is equivalent to the corresponding one in Eqs.~(\meqref{eq:3.3})-(\meqref{eq:3.20}) in taking $A=A_v,B=(A^*)_v$ and
$$l_{\wedge_A}=-R_{\swarrow}^*,r_{\wedge_A}=L_{\succ}^*,
l_{\vee_A}=R_{\prec}^*,r_{\vee_A}=-L_{\nearrow}^*,
l_{\wedge_B}=-R_{\swarrow_*}^*,r_{\wedge_B}=L_{\succ_*}^*,
l_{\vee_B}=R_{\prec_*}^*,r_{\vee_B}=-L_{\nearrow_*}^*.$$ As an
example, we give an explicit proof that
\begin{equation} L_{\succ_*}^*(a^*)(x\wedge y)-
L_{\succ_*}^*(R_{\nwarrow}^*(y)a^*)x-x\wedge(L_{\searrow_*}^*(a^*)y)=0,
\mlabel{eq:5.19}
\end{equation}
holds if and only if Eq.~(\mref{eq:5.3}) holds. The other
equivalences are similar. In fact, let the left hand side of
Eq.~(\meqref{eq:5.19}) act on an arbitrary element $b^*\in A^*$.
Then we have
\begin{eqnarray*}
& &\langle L_{\succ_*}^*(a^*)(x\wedge y)-
L_{\succ_*}^*(R_{\nwarrow}^*(y)a^*)x-x\wedge(L_{\searrow_*}^*(a^*)y),b^*\rangle \\
&=&\langle x\wedge y, a^*\succ_*b^*\rangle - \langle
x,(R_{\nwarrow}^*(y)a^*)\succ_*b^*\rangle -\langle y,
 a^*\searrow_*(L_{\wedge}^*(x)b^*)\rangle \\
 &=&\langle (\beta+\tilde{\beta})(x\wedge y)-
 (R_{\nwarrow}(y)\otimes 1)(\beta+\tilde{\beta})(x)-
 (1\otimes L_{\wedge}(x))\tilde{\beta}(y),a^*\otimes b^*\rangle .
\end{eqnarray*}
So the conclusion follows.
\end{proof}

\begin{defn}
\begin{enumerate}
\item Let $(A,\alpha,\beta,\tilde{\alpha},\tilde{\beta})$ be a
vector space with four comultiplications
$\alpha,\beta,\tilde{\alpha},\tilde{\beta}: A\rightarrow A\otimes
A$. Let the dual operations be
$\alpha^*,\beta^*,\tilde{\alpha}^*,\tilde{\beta}^*:A^*\otimes
A^*\subset(A\otimes A)^*\rightarrow A^*$. If
$(A^*,\alpha^*,\beta^*,\tilde{\alpha}^*,\tilde{\beta}^*)$ becomes
a quadri-algebra, then we call
$(A,\alpha,\beta,\tilde{\alpha},\tilde{\beta})$ a {\bf
quadri-coalgebra}. A {\bf homomorphism} between two
quadri-coalgebras is defined as a linear map (between the two
underlying vector spaces) which preserves the corresponding
cooperations. \item Let
$(A,\nwarrow,\nearrow,\swarrow,\searrow,\alpha,\beta,\tilde{\alpha},\tilde{\beta})$
be a quadri-algebra equipped with four cooperations
$\alpha,\beta,\tilde{\alpha}$, $\tilde{\beta}:A\rightarrow
A\otimes A$ such that
$(A,\alpha,\beta,\tilde{\alpha},\tilde{\beta})$ is a
quadri-coalgebra and $\alpha,\beta,\tilde{\alpha}$ and
$\tilde{\beta}$ satisfy Eqs.~(\ref{eq:5.1})-(\ref{eq:5.18}). Then
$(A,\nwarrow,\nearrow,\swarrow,\searrow,
\alpha,\beta,\tilde{\alpha},\tilde{\beta})$ is called a {\bf
quadri-bialgebra}. A {\bf homomorphism} between two
quadri-bialgebras is defined as a linear map (between the two
underlying vector spaces) which is a homomorphisms of both
quadri-algebras and quadri-coalgebras.
\end{enumerate}
\mlabel{de:quadbial}
\end{defn}

Combining Proposition~\mref{pp:5.1} and the discussion in the previous sections, we have the following conclusion:

\begin{theorem}  Let
$(A,\nwarrow_A,\nearrow_A,\swarrow_A,\searrow_A,\alpha,\beta,
\tilde{\alpha},\tilde{\beta})$ be a quadri-algebra with four
comultiplications
$\alpha,\beta,\tilde{\alpha},\tilde{\beta}:A\rightarrow A\otimes
A$, such that $(A,\alpha,\beta,\tilde{\alpha},\tilde{\beta})$ is a
quadri-coalgebra. Then, with the notations
$\nwarrow_B:=\alpha^*,\nearrow_B:=\beta^*,
\swarrow_B:=\tilde{\alpha}^*,\searrow_B:=\tilde{\beta}^*$, and
with $\mathfrak{B}_S$ given by Eq.~(\meqref{eq:3.2}), the
following conditions are equivalent

\begin{enumerate}
\item $(A_v\bowtie(A^*)_v,A_v,(A^*)_v, \mathfrak{B}_S)$ is a
standard Manin triple of dendriform algebras associated to the
nondegenerate 2-cocycle $\frak B_S$; \mlabel{it:5.3a} \item
$(A\bowtie A^*,A,A^*,\mathfrak{B}_S)$ is a standard Manin triple
of quadri-algebras associated to the nondegenerate invariant
bilinear form $\frak B_S$. \mlabel{it:5.3b} \item
$(A_v,(A^*)_v,-R_{\swarrow_A}^*,L_{\succ_A}^*,R_{\prec_A}^*,
-L_{\nearrow_A}^*,-R_{\swarrow_{A^*}}^*,L_{\succ_{A^*}}^*,R_{\prec_{A^*}}^*,
-L_{\nearrow_{A^*}}^*)$ is a matched pair of dendriform algebras;
\mlabel{it:5.3c} \item
$(A,A^*,R_{\searrow_A}^*,L_{\star_A}^*,-R_{\vee_A}^*,
-L_{\prec_A}^*,-R_{\succ_A}^*,-L_{\wedge_A}^*,R_{\star_A}^*,
L_{\nwarrow_A}^*,R_{\searrow_{A^*}}^*,L_{\star_{A^*}}^*,-R_{\vee_{A^*}}^*,
-L_{\prec_{A^*}}^*$, $-R_{\succ_{A^*}}^*$, $-L_{\wedge_{A^*}}^*$,
$R_{\star_{A^*}}^*, L_{\nwarrow_{A^*}}^*)$ is a matched pair of
quadri-algebras; \mlabel{it:5.3d} \item
$(A,\nwarrow_A,\nearrow_A,\swarrow_A,
\searrow_A,\alpha,\beta,\tilde{\alpha},\tilde{\beta})$ is a
quadri-bialgebra. \mlabel{it:5.3e}
\end{enumerate}
\mlabel{thm:5.3}
\end{theorem}

By a standard argument (cf. \cite{Bai2}, Proposition 2.2.10), we get the following result:

\begin{prop} Two Manin triples of dendriform
algebras associated to a nondegenerate 2-cocycle are isomorphic if
and only if their corresponding quadri-bialgebras are isomorphic.
\mlabel{pp:5.4}
\end{prop}

\begin{remark}
{\rm It is obvious that for a quadri-bialgebra
$(A,\nwarrow,\nearrow,\swarrow,
\searrow,\alpha,\beta,\tilde{\alpha},\tilde{\beta})$, the dual
$(A^*$, $\alpha^*$, $\beta^*$,
$\tilde{\alpha}^*,\tilde{\beta}^*,\gamma,\delta, \tilde{\gamma}$,
$\tilde{\delta})$ is also a quadri-bialgebra, where
$\gamma^*=\nwarrow, \delta^*=\nearrow,
 \tilde{\gamma}^*=\swarrow$, $\tilde{\delta}^*=\searrow$. } \mlabel{rk:5.5}
\end{remark}

\section{Coboundary quadri-bialgebras}
\mlabel{sec:cob}

We now give a more detailed study of a class of quadri-bialgebras
with certain coboundary conditions which can be obtained as
solutions of a system of equations.

\begin{defn}{\rm
 A quadri-bialgebra
$(A,\nwarrow,\nearrow,\swarrow,\searrow,\alpha,\beta,
\tilde{\alpha},\tilde{\beta})$ is called {\bf coboundary} if
$\alpha,\beta,\tilde{\alpha}$ and $\tilde{\beta}$ satisfy
\begin{eqnarray}
\alpha(x)&=&(-1\otimes L_{\searrow}(x)+R_{\star}(x)\otimes
1)r_{\nwarrow}, \mlabel{eq:6.1}\\ 
\beta(x)&=&(1\otimes L_{\vee}(x)- R_{\prec}(x)\otimes
1)r_{\nearrow}, \mlabel{eq:6.2}\\
\tilde{\alpha}(x)&=&(1\otimes L_{\succ}(x)- R_{\wedge}(x)\otimes
1)r_{\swarrow}, \mlabel{eq:6.3}\\
\tilde{\beta}(x)&=&(-1\otimes L_{\star}(x)+ R_{\nwarrow}(x)\otimes
1)r_{\searrow}, \mlabel{eq:6.4}
\end{eqnarray}
where $r_{\nwarrow},r_{\nearrow},r_{\swarrow},r_{\searrow}\in
A\otimes A$ and $x\in A$.}
\mlabel{de:6.1}
\end{defn}

\begin{prop} Let $(A,\nwarrow,\nearrow,\swarrow,\searrow,\alpha,\beta,
\tilde{\alpha},\tilde{\beta})$ be a quadri-algebra with
 four comultiplications $\alpha,\beta,
\tilde{\alpha},\tilde{\beta}$ defined by
Eqs.~(\ref{eq:6.1})-(\ref{eq:6.4})
 respectively. If $r_{\nwarrow}=r_{\nearrow}=r_{\swarrow}=
 r_{\searrow}=r\in A\otimes A$ and
  $r$ is skew-symmetric, then $\alpha,\beta,
\tilde{\alpha}$ and $\tilde{\beta}$ satisfy
Eqs.~(\meqref{eq:5.1})-(\meqref{eq:5.18}). \mlabel{pp:6.2}
\end{prop}

\begin{proof}
 Straightforward. \end{proof}

\begin{lemma} Let $A$ be a vector space and
$\alpha,\beta,\tilde{\alpha},\tilde{\beta}:A\rightarrow A\otimes
A$ be four cooperations. Then
$(A,\alpha,\beta,\tilde{\alpha},\tilde{\beta})$ is a
quadri-coalgebra if and only if the linear maps $R_i:A\rightarrow
A\otimes A\otimes A$ $(i\in \{1,...,9\})$ defined by the following
equations are zero: \begin{eqnarray}R_1(x):&=&(\alpha\otimes
1)\alpha(x)-(1\otimes(\alpha+\beta+\tilde{\alpha}+\tilde{\beta}))\alpha(x),
\mlabel{eq:6.5}\\
R_2(x):&=&(\beta\otimes
1)\alpha(x)-(1\otimes(\alpha+\tilde{\alpha}))\beta(x),
\mlabel{eq:6.6}\\
R_3(x):&=&((\alpha+\beta)\otimes
1)\beta(x)-(1\otimes(\beta+\tilde{\beta}))\beta(x),
\mlabel{eq:6.7}\\
R_4(x):&=&(\tilde{\alpha}\otimes
1)\alpha(x)-(1\otimes(\alpha+\beta))\tilde{\alpha}(x),
\mlabel{eq:6.8}\\
R_5(x):&=&(\tilde{\beta}\otimes
1)\alpha(x)-(1\otimes\alpha)\tilde{\beta}(x), \mlabel{eq:6.9}\\
R_6(x):&=&((\tilde{\alpha}+\tilde{\beta})\otimes
1)\beta(x)-(1\otimes\beta)\tilde{\beta}(x), \mlabel{eq:6.10}\\
R_7(x):&=&((\alpha+\tilde{\alpha})\otimes
1)\tilde{\alpha}(x)-(1\otimes(\tilde{\alpha}+
\tilde{\beta}))\tilde{\alpha}(x), \mlabel{eq:6.11}\\
R_8(x):&=&((\beta+\tilde{\beta})\otimes
1)\tilde{\alpha}(x)-(1\otimes\tilde{\alpha})\tilde{\beta}(x),
\mlabel{eq:6.12}\\
R_9(x):&=&((\alpha+\beta+\tilde{\alpha}+\tilde{\beta})\otimes1)\tilde{\beta}(x)-
(1\otimes\tilde{\beta})\tilde{\beta}(x),\mlabel{eq:6.13}
\end{eqnarray}
for all $x\in A$.
\mlabel{lem:6.3}
\end{lemma}

\begin{proof} It follows from the definition of a
quadri-algebra.\end{proof}

\begin{defn}{\rm  Let
$(A,\nwarrow,\nearrow,\swarrow,\searrow)$ be a quadri-algebra and
let $r\in A\otimes A$. The following equations are called the {\bf
$Q_i^j$-equations} for $i=1,2,3$ and $j=1,2$:
\begin{equation}Q_1^1:=r_{23}\wedge r_{12}-r_{13}\succ r_{23}+
r_{12}\swarrow
r_{13}=0,
\mlabel{eq:6.14}
\end{equation}\begin{equation}Q_1^2:=r_{23}\vee
r_{12}-r_{12}\prec r_{13}+ r_{13}\nearrow
r_{23}=0,
\mlabel{eq:6.15}
\end{equation}\begin{equation}Q_2^1:=r_{12}\wedge
r_{13}-r_{23}\succ r_{12}- r_{13}\swarrow
r_{23}=0,\mlabel{eq:6.16}
\end{equation}\begin{equation}Q_2^2:=r_{12}\vee
r_{13}+r_{13}\prec r_{23}+ r_{23}\nearrow
r_{12}=0,\mlabel{eq:6.17}
\end{equation}\begin{equation}Q_3^1:=r_{13}\wedge
r_{23}+r_{12}\succ r_{13}+ r_{23}\swarrow
r_{12}=0,\mlabel{eq:6.18}
\end{equation}\begin{equation}Q_3^2:=r_{13}\vee
r_{23}-r_{23}\prec r_{12}- r_{12}\nearrow r_{13}=0.\mlabel{eq:6.19}
\end{equation}
}
\mlabel{de:6.4}
\end{defn}

\begin{prop} Let $(A,\nwarrow,\nearrow,\swarrow,\searrow)$ be a
quadri-algebra and let $r\in A\otimes A$. Define
$\alpha,\beta,\tilde{\alpha}$ and $\tilde{\beta}$ by Eqs.~
(\ref{eq:6.1})-(\ref{eq:6.4}) respectively, where
$r_{\nwarrow}=r_{\nearrow}=r_{\swarrow}=r_{\searrow}=r$. Then
$(A,\alpha,\beta,\tilde{\alpha},\tilde{\beta})$ becomes a
quadri-coalgebra if and only if the following equations hold (for
all $x\in A$):
\begin{equation}(1\otimes 1\otimes L_{\searrow}(x)-
R_{\star}(x)\otimes 1\otimes
1)(Q_1^2-Q_3^1)=0,\mlabel{eq:6.20}
\end{equation}\begin{equation}(1\otimes 1\otimes
L_{\searrow}(x)- R_{\prec}(x)\otimes 1\otimes
1)Q_1^2=0,\mlabel{eq:6.21}\end{equation}\begin{equation}(1\otimes 1\otimes
L_{\vee}(x)- R_{\prec}(x)\otimes 1\otimes
1)Q_3^1=0,\mlabel{eq:6.22}
\end{equation}\begin{equation}(1\otimes 1\otimes
L_{\searrow}(x)- R_{\wedge}(x)\otimes 1\otimes
1)Q_2^1=0,\mlabel{eq:6.23}
\end{equation}\begin{equation}(1\otimes 1\otimes
L_{\searrow}(x)- R_{\nwarrow}(x)\otimes 1\otimes
1)(Q_2^1+Q_3^2)=0,\mlabel{eq:6.24}
\end{equation}\begin{equation}(1\otimes 1\otimes
L_{\vee}(x)- R_{\nwarrow}(x)\otimes 1\otimes
1)Q_3^2=0,\mlabel{eq:6.25}
\end{equation}\begin{equation}(1\otimes 1\otimes
L_{\succ}(x)- R_{\wedge}(x)\otimes 1\otimes
1)Q_2^2=0,\mlabel{eq:6.26}
\end{equation}\begin{equation}(1\otimes 1\otimes
L_{\succ}(x)- R_{\nwarrow}(x)\otimes 1\otimes
1)Q_1^1=0,\mlabel{eq:6.27}
\end{equation}\begin{equation}(1\otimes 1\otimes
L_{\star}(x)- R_{\nwarrow}(x)\otimes 1\otimes
1)(Q_3^1+Q_3^2)=0.\mlabel{eq:6.28}
\end{equation}
\mlabel{pp:6.5}
\end{prop}

\begin{proof} Let $x\in A$. We only need to prove that $R_i(x)=0, 1\leq i\leq 9$, if and
only if the corresponding equation in
Eqs.~(\meqref{eq:6.20})-(\meqref{eq:6.28}) holds. Since all the
proofs are similar, we only give a detailed proof for that
$R_2(x)=0$ if and only if Eq.~(\meqref{eq:6.21}) holds. In fact,
after rearranging the terms suitably, we divide $R_2(x)$ into
three parts: $R_2(x)=(R1)+(R2)+(R3),$ where
\begin{eqnarray*}
(R1)&=&\sum_{i,j}\{-a_i\otimes a_j\vee b_i\otimes x\searrow b_j+
a_i\prec a_j\otimes b_i\otimes x\searrow b_j+a_j\otimes
a_i\otimes(x\vee b_j)\searrow b_i\\&&-a_j\otimes a_i\otimes(x\vee
b_j)\succ b_i\}\\
&=&(1\otimes 1\otimes L_{\searrow}(x))(-r_{23}\vee
r_{12}+r_{12}\prec r_{13}-r_{13}\nearrow r_{23}),\\
(R2)&=&\sum_{i,j}\{a_i\otimes(a_j\star x)\vee b_i\otimes
b_j-a_j\otimes a_i\star(x\vee b_j)\otimes b_i+a_j\otimes
a_i\wedge(x\vee b_j)\otimes b_i\}=0,\\
(R3)&=&\sum_{i,j}\{-a_i\prec(a_j\star x)\otimes b_i\otimes
b_j-a_j\prec x\otimes a_i\otimes b_j\searrow b_i+a_j\prec x\otimes
a_i\star b_j\otimes b_i\\&
&+a_j\prec x\otimes a_i\otimes b_j\succ b_i-a_j\prec x\otimes a_i\wedge b_j\otimes b_i\\
&=&(R_{\prec}(x)\otimes 1\otimes 1)(-r_{12}\prec
r_{13}+r_{13}\nearrow r_{23}+r_{23}\vee r_{12}).
\end{eqnarray*} So the conclusion follows.\end{proof}

From the above discussion, we arrive at the following conclusion.

\begin{theorem}
Let $(A,\nwarrow,\nearrow,\swarrow,\searrow)$ be a quadri-algebra
and let $r\in A\otimes A$ be skew-symmetric. Then the
comultiplications $\alpha,\beta,\tilde{\alpha}$ and
$\tilde{\beta}$ defined by
Eqs.~(\meqref{eq:6.1})-(\meqref{eq:6.4}) with
$r_{\nwarrow}=r_{\nearrow}=r_{\swarrow}=r_{\searrow}= r$
respectively make
$(A,\nwarrow,\nearrow,\swarrow,\searrow,\alpha,\beta,
\tilde{\alpha},\tilde{\beta})$ into a quadri-bialgebra if and only
if Eqs.~(\meqref{eq:6.20})-(\meqref{eq:6.28}) are satisfied.
\mlabel{thm:6.6}
\end{theorem}

Moreover, we have the following ``Drinfeld double" construction of
a quadri-bialgebra~\cite{CP}.

\begin{theorem}
 Let $(A,\nwarrow,\nearrow,\swarrow,\searrow,\alpha,\beta,
\tilde{\alpha},\tilde{\beta})$ be a  quadri-bialgebra. Then there
exists a canonical quadri-bialgebra structure on $A\oplus A^*$
such that the inclusion $i_1:A\rightarrow A\oplus A^*$ is a
homomorphism of quadri-bialgebras, which is from
$(A,\nwarrow,\nearrow,\swarrow,\searrow,-\alpha,-\beta,
-\tilde{\alpha},-\tilde{\beta})$ to $A\oplus A^*$, and the
inclusion $i_2:A^*\rightarrow A\oplus A^*$ is also a homomorphism
of quadri-bialgebras from the quadri-bialgebra given in
Remark~\ref{rk:5.5} to $A\oplus A^*$. \mlabel{thm:6.7}
\end{theorem}

\begin{proof} Let $r=\sum_ie_i\otimes e_i^*$ correspond to the identity map ${\rm id}:A\rightarrow
A$, where $\{e_i,...,e_s\}$ is a basis of $A$ and
$\{e_1^*,...e_s^*\}$ is the dual basis.  Suppose the
quadri-algebra structure
$(\nwarrow_{\bullet},\nearrow_{\bullet},\swarrow_{\bullet},\searrow_{\bullet})$
on $A\oplus A^*$ is given by $\mathcal{QD}(A):=
A\bowtie_{R_{\searrow}^*,L_{\star}^*,-R_{\vee}^*,
-L_{\prec}^*,-R_{\succ}^*,-L_{\wedge}^*,R_{\star}^*,L_{\nwarrow}^*}^{R_{\searrow_*}^*,
L_{\star_*}^*,-R_{\vee_*}^*,-L_{\prec_*}^*,-R_{\succ_*}^*,-L_{\wedge_*}^*,R_{\star_*}^*,
L_{\nwarrow_*}^*}A^*$, where the subscript $\ast$ is used to
denote the quadri-algebra structure on $A^*$. Then for all $x,y\in
A,a^*,b^*\in A^*$,

{\small$$
x\nwarrow_{\bullet}y=x\nwarrow y,
x\nearrow_{\bullet}y=x\nearrow y,   x\swarrow_{\bullet}y=x\swarrow
y, x\searrow_{\bullet}y=x\searrow y,$$
$$
\smallskip
a^*\nwarrow_{\bullet}b^*=a^*\nwarrow_*b^*,
a^*\nearrow_{\bullet}b^*=a^*\nearrow_*b^*,
a^*\swarrow_{\bullet}b^*=a^*\swarrow_*b^*,
a^*\searrow_{\bullet}b^*=a^*\searrow_*b^*,$$
$$\smallskip
x\nwarrow_{\bullet}a^*=R_{\searrow}^*(x)a^*+L_{\star_*}^*(a^*)x,
x\nearrow_{\bullet}a^*=-R_{\vee}^*(x)a^*-L_{\prec_*}^*(a^*)x,
x\swarrow_{\bullet}a^*=-R_{\succ}^*(x)a^*-L_{\wedge_*}^*(a^*)x,$$
$$\smallskip
x\searrow_{\bullet}a^*=R_{\star}^*(x)a^*+L_{\nwarrow_*}^*(a^*)x,
a^*\nwarrow_{\bullet}x=R_{\searrow_*}^*(a^*)x+L_{\star}^*(x)a^*,
a^*\nearrow_{\bullet}x=-R_{\vee^*}^*(a^*)x-L_{\prec}^*(x)a^*,$$
$$\smallskip
a^*\swarrow_{\bullet}x=-R_{\succ^*}^*(a^*)x-L_{\wedge}^*(x)a^*,
a^*\searrow_{\bullet}x=R_{\star_*}^*(a^*)x+L_{\nwarrow}^*(x)a^*.$$}
We prove that $r$ satisfies
Eqs.~(\meqref{eq:5.1})-(\meqref{eq:5.18}). We provide a detailed
proof that $r$ satisfies Eq.~(\meqref{eq:5.13}) as an example. The
proofs of the other cases follow from the same argument.

for all $\mu,\nu\in\mathcal{QD}(A)$, Eq.~(\ref{eq:5.13}) is
equivalent to
\begin{eqnarray*}
& &\sum_k\{-(\mu\wedge_{\bullet}e_k^*)\wedge_{\bullet}\nu\otimes
e_k-e_k^*\wedge_{\bullet}\nu\otimes e_k\swarrow_{\bullet}\mu+
\mu\wedge_{\bullet}e_k^*\otimes\nu\searrow_{\bullet}e_k+
e_k^*\otimes\nu\searrow_{\bullet}(e_k\swarrow_{\bullet}\mu)\\&&+\mu\wedge
e_k\otimes\nu\searrow_{\bullet}e_k^*-\mu\wedge_{\bullet}(e_k\star_{\bullet}\nu)\otimes
e_k^*+ e_k\otimes(\nu\succ_{\bullet}e_k^*)\swarrow_{\bullet}\mu-
e_k\wedge_{\bullet}\nu\otimes e_k^*\swarrow_{\bullet}\mu\}=0.
\end{eqnarray*}
The proof of the equation is divided into the following four
cases: (I). $\mu,\nu\in A$; (II). $\mu,\nu\in A^*$; (III). $\mu\in
A, \nu\in A^*$; (IV). $\mu\in A^*,\nu\in A$. We proceed to prove
the last case as the proofs of the other cases are similar. Let
$\mu=e_i^*, \nu=e_j$. Then for all $m$ and $n$, the coefficient of
$e_m^*\otimes e_n$ is
\begin{eqnarray*}
& &\sum_k\langle
-(e_i^*\wedge_{\bullet}e_n^*)\wedge_{\bullet}e_j,e_m\rangle -
\langle e_k^*\wedge_{\bullet}e_j,e_m\rangle \langle
e_k\swarrow_{\bullet}e_i^*,e_n^*\rangle +
\langle e_i^*\wedge_{\bullet}e_k^*,e_m\rangle \langle e_j\searrow_{\bullet}e_k,e_n^*\rangle \\
& & +\langle
e_j\searrow_{\bullet}(e_m\swarrow_{\bullet}e_i^*),e_n^*\rangle +
\langle e_i^*\wedge_{\bullet}e_k,e_m\rangle \langle e_j\searrow_{\bullet}e_k^*,e_n^*\rangle \\
&=&\sum_k-\langle e_i^*\wedge_*e_n^*,e_j\succ e_m\rangle +\langle
e_k^*,e_j\succ e_m\rangle \langle e_k,e_i^*\wedge_*e_n^*\rangle +
\langle e_i^*\wedge_*e_k^*,e_m\rangle \langle e_j\searrow
e_k,e_n^*\rangle \\
& &-\langle e_i^*,e_k\succ e_m\rangle \langle
e_j,e_k^*\nwarrow_*e_n^*\rangle - \langle
e_m,e_i^*\wedge_*e_k^*\rangle \langle e_j\searrow e_k,e_n^*\rangle +
\langle e_i^*,e_k\succ e_m\rangle \langle
e_j,e_k^*\nwarrow_*e_n^*\rangle\\&=&0.
\end{eqnarray*} Similarly, the coefficients of $e_m\otimes e_n,
e_m\otimes e_n^*$ and  $e_m^*\otimes e_n^*$ are zero.

Furthermore, we have
\begin{eqnarray*}
r_{23}\wedge_{\bullet}r_{12}&=&\sum_{i,j}e_j\otimes
e_i\wedge_{\bullet}e_j^*\otimes e_i^*\\
&=&
\sum_{i,j}e_j\otimes[-R_{\swarrow}^*(e_i)e_j^*+
L_{\succ_*}^*(e_j^*)e_i]\otimes e_i^*\\
&=&\sum_{i,j,k}-e_j\otimes e_k^*\langle e_j^*,e_k\swarrow e_i\rangle
\otimes
e_i^*+e_j\otimes e_k\langle e_i,e_j^*\succ_*e_k^*\rangle \otimes e_i^*\\
&=&\sum_{i,k}-e_k\swarrow e_i\otimes e_k^*\otimes e_i^* +e_i\otimes
e_k\otimes e_i^*\succ_*e_k^*\\
&=&-r_{12}\swarrow_{\bullet}r_{13}+r_{13}\succ_{\bullet}r_{23}.
\end{eqnarray*}
So $Q_1^1=0$. Similarly, $r$ satisfies Eqs.
(\mref{eq:6.15})-(\mref{eq:6.19}). So the cooperations (for all
$x\in A\oplus A^*$)
$$\alpha_{\mathcal{QD}}(x)=(-1\otimes L_{\searrow}(x)+R_{\star}(x)\otimes
1)r, \;\;  \beta_{\mathcal{QD}}(x)=(1\otimes L_{\vee}(x)-
R_{\prec}(x)\otimes 1)r,$$
$$\tilde{\alpha}_{\mathcal{QD}}(x)=(1\otimes L_{\succ}(x)-
R_{\wedge}(x)\otimes 1)r,
\;\;\tilde{\beta}_{\mathcal{QD}}(x)=(-1\otimes L_{\star}(x)+
R_{\nwarrow}(x)\otimes 1)r$$ induce a quadri-bialgebra structure
on $A\oplus A^*$.

On the other hand, for $e_i\in A$, we have:
\begin{eqnarray*}
\alpha_{\mathcal{QD}}(e_i)
&=&\sum_{j,k}\{-e_j\otimes(\langle e_j^*,e_k\star e_i\rangle e_k^*+
\langle e_j^*\nwarrow_*e_k^*,e_i\rangle e_k)+e_j\star e_i\otimes e_j^*\}\\
&=&-\sum_{j,k}\langle e_j^*\nwarrow_*e_k^*,e_i\rangle e_j\otimes
e_k\\&=&-\alpha(e_i).
\end{eqnarray*}
Similarly, we have $\beta_{\mathcal{QD}}(e_i)=-\beta(e_i),
\tilde{\alpha}_{\mathcal{QD}}(e_i)=-\tilde{\alpha}(e_i)$ and
$\tilde{\beta}_{\mathcal{QD}}(e_i)=-\tilde{\beta}(e_i)$. So $i_1$ is
a homomorphism of quadri-bialgebras. Similarly, we prove that $i_2$
is also a homomorphism of quadri-bialgebras.\end{proof}

\begin{defn}
{\rm \quad Let
$(A,\nwarrow,\nearrow,\swarrow,\searrow,\alpha,\beta,
\tilde{\alpha},\tilde{\beta})$ be a quadri-bialgebra. With the
quadri-bialgebra structure given in Theorem~\mref{thm:6.7},
$A\oplus A^*$ is called the {\bf Drinfeld $Q$-double} of $A$ and
we denote it by $\mathcal{QD}(A)$. On the other hand, due to the
symmetry between $A$ and $A^*$ (Remark~\mref{rk:5.5}),
$\tilde{r}:=\sum_ie_i^*\otimes e_i$ also induces a (coboundary)
quadri-bialgebra structure on $A\oplus A^*$, and we denote it by
$\tilde{\mathcal{QD}}(A)$.} \mlabel{de:6.8}
\end{defn}

\begin{prop}Let
$(A,\nwarrow,\nearrow,\swarrow,\searrow,\alpha,\beta,
\tilde{\alpha},\tilde{\beta})$ be a quadri-bialgebra. Suppose that
$\alpha,\beta,\tilde{\alpha}$ and $\tilde{\beta}$ are defined by
Eqs.~(\meqref{eq:6.1})-(\meqref{eq:6.4}) with
$r_{\nwarrow}=r_{\nearrow} =r_{\swarrow}=r_{\searrow}=r$
respectively.

\begin{enumerate}
\item If $r$ satisfies Eqs.~(\meqref{eq:6.14})-(\meqref{eq:6.19}),
then $T_r$ is a homomorphism of quadri-bialgebras which is from
the quadri-bialgebra given in Remark~\mref{rk:5.5} to
$(A,\nwarrow,\nearrow$, $\swarrow,\searrow,-\alpha,-\beta$,
$-\tilde{\alpha}$, $-\tilde{\beta})$. \mlabel{it:6.9a} \item If
$r$ satisfies Eqs.~(\meqref{eq:6.14})-(\meqref{eq:6.19}) and $r$
is skew-symmetric, then
\begin{equation}
\tilde{T}_r(x+a^*):=x+T_r(a^*),\forall x\in A,a^*\in A^*,
\end{equation} is a homomorphism of quadri-bialgebras from $\tilde{\mathcal{QD}}(A)$ to
$(A,\nwarrow,\nearrow,\swarrow,\searrow,\alpha,\beta,
\tilde{\alpha},\tilde{\beta})$.
\mlabel{it:6.9b}
\end{enumerate}
\mlabel{pp:6.9}
\end{prop}

\begin{proof} (\meqref{it:6.9a}). Note that
$(1\otimes\alpha)r=r_{12}\nwarrow r_{13},(\alpha\otimes 1)r=
-r_{13}\nwarrow r_{23}$. Set $\nwarrow_*:=\alpha^*$. Then for all
$a^*,b^*\in A^*$ we have
\begin{eqnarray*}
T_r(a^*\nwarrow_*b^*)&=&\langle 1\otimes(a^*\nwarrow_*b^*),r\rangle\\
&=&\langle 1\otimes a^*\otimes b^*,(1\otimes\alpha)r\rangle\\
& =&\langle 1\otimes a^*\otimes b^*,r_{12}\nwarrow r_{13}\rangle\\
&=&T_r(a^*)\nwarrow T_r(b^*),\\
(T_r\otimes T_r)\gamma(a^*)&=&T_r(a_{(1)}^*)\otimes
T_r(a_{(2)}^*)\\
&=&\sum_iu_i\otimes u_j\langle a_{(1)}^*,v_i\rangle
\langle a_{(2)}^*,v_j\rangle\\
& =& \langle 1\otimes 1\otimes
a^*,r_{13}\nwarrow r_{23}\rangle\\
& =&-(1\otimes 1\otimes
a^*)(\alpha\otimes 1)r\\
&=&-\alpha(T_r(a^*)).
\end{eqnarray*}
Here we use
the Sweedler's notation: $\gamma(a^*)=a_{(1)}^*\otimes a_{(2)}^*$.
Similarly $T_r$ also preserves the other operations and cooperations. So
the conclusion follows.

(\meqref{it:6.9b}). We still denote the products of
$\tilde{\mathcal{QD}}(A)$ by
$\nwarrow,\nearrow,\swarrow,\searrow$. First we prove that
$\tilde{T}_r$ is a homomorphism of quadri-algebras, that is,
$\tilde{T}_r(\mu\diamond\nu)=\tilde{T}_r(\mu)\diamond\tilde{T}_r(\nu)$
for all $\mu,\nu\in A\oplus A^*$ and
$\diamond\in\{\nwarrow,\nearrow,\swarrow,\searrow\}$. It is
obvious when $\mu,\nu\in A$. Moreover, for all $x\in A,a^*\in A^*$
we have
\begin{eqnarray*}
\tilde{T}_r(x\nwarrow
a^*)&=&\tilde{T}_r(R_{\searrow}^*(x)a^*+L_{\star}^*(a^*)x)\\
&=&T_r(R_{\searrow}^*(x)a^*)+L_{\star}^*(a^*)x\\
&=&(1\otimes a^*)((1\otimes R_{\searrow}(x))r)+(a^*\otimes
1)((-1\otimes L_{\nwarrow}(x)+R_{\searrow}(x)\otimes 1)r)\\
&=&(1\otimes a^*)((L_{\nwarrow}(x)\otimes 1)r)\\
&=&\tilde{T}_r(x)\nwarrow\tilde{T}_r(a^*),
\end{eqnarray*}
where we use the fact that $r$ is skew-symmetric. Similarly,
$\tilde{T}_r(a^*\nwarrow
x)=\tilde{T}_r(a^*)\nwarrow\tilde{T}_r(x)$ and, for each
$\diamond\in\{\nearrow,\swarrow,\searrow\}$, we have
$\tilde{T}(\mu\diamond\nu)=\tilde{T}_r(\mu)\diamond\tilde{T}_r(\nu)$
for all $\mu\in A,\nu\in A^*$ or $\mu\in A^*,\nu\in A$. On the
other hand, by Item~(\meqref{it:6.9a}), for all $a^*,b^*\in A^*$,
we have
$$\tilde{T}_r(a^*\diamond b^*)=T_r(a^*\diamond b^*)=T_r(a^*)\diamond
T_r(b^*)=\tilde{T}_r(a^*)\diamond \tilde{T}_r(b^*),$$ where
$\diamond\in\{\nwarrow,\nearrow,\swarrow,\searrow\}$. Therefore,
$\tilde{T}_r$ is a homomorphism of quadri-algebras.

Furthermore, let $\{e_1,...,e_n\}$ be a basis of $A$ and let
$\{e_1^*,...,e_n^*\}$ be the dual basis. Define
$\tilde{r}:=\sum_ie_i^*\otimes e_i$. Then we have
$$(\tilde{T}_r\otimes\tilde{T}_r)(\tilde{r})=\sum_i\tilde{T}_r(e_i^*)\otimes\tilde{T}_r(e_i)=
\sum_i(1\otimes e_i^*)(r)\otimes e_i=r,$$ which implies that
$\tilde{T}_r$ is a homomorphism of quadri-coalgebras. So the
conclusion follows.\end{proof}

\section{The $Q$-equations}
\mlabel{sec:qeq}

We now show that the equations in Definition~\mref{de:6.4} boil
down to a pair of equations, called \qeq which have a clear
connection with $\calo$-operators.

 By Theorem~\mref{thm:6.6},  we
have the following conclusion.

\begin{prop}
 Let
$(A,\nwarrow,\nearrow,\swarrow,\searrow)$ be a quadri-algebra and
$r\in A\otimes A$ be skew-symmetric. Then the comultiplications
$\alpha,\beta,\tilde{\alpha}$ and $\tilde{\beta}$ defined by
Eqs.~(\meqref{eq:6.1})-(\meqref{eq:6.4}) with
$r_{\nwarrow}=r_{\nearrow} =r_{\swarrow}=r_{\searrow}=r$
respectively make
$(A,\nwarrow,\nearrow,\swarrow,\searrow,\alpha,\beta,
\tilde{\alpha},\tilde{\beta})$ into a quadri-bialgebra if $r$
satisfies Eqs.~(\meqref{eq:6.14})-(\meqref{eq:6.19}).
\mlabel{pp:7.1}
\end{prop}

As a related notion, we recall
\begin{defn}{\rm \cite{Bai3} Let $(A,\nwarrow,\nearrow$, $\swarrow,\searrow)$ be a
quadri-algebra and let $(V,l_{\nwarrow},r_{\nwarrow},l_{\nearrow},
r_{\nearrow}$, $l_{\swarrow},r_{\swarrow}$,
$l_{\searrow},r_{\searrow})$ be a bimodule. An {\bf
$\mathcal{O}$-operator} of $A$ associated to the bimodule $V$ is a
linear map $T$ from $V$ to $A$ such that for all $u,v\in V$,
\begin{equation}T(u)\circ T(v)=
T(l_{\circ}(T(u))v+ r_{\circ}(T(v))u),
\circ\in\{\nwarrow,\nearrow,\swarrow,\searrow\}.
\mlabel{eq:7.1}
\end{equation}}
\label{de:7.2}
\end{defn}

In fact, the following relation holds.
\begin{prop} {\rm \cite{Bai3}} Let $(A,\nwarrow,\nearrow,\swarrow,\searrow)$ be a
quadri-algebra and let $r\in A\otimes A$ be skew-symmetric. Then the
following conditions are equivalent:
\begin{enumerate}
\item $r$ satisfies Eqs.~(\meqref{eq:6.14})-(\meqref{eq:6.15});
\mlabel{it:7.3a} \item $r$ satisfies
Eqs.~(\meqref{eq:6.16})-(\meqref{eq:6.17}); \mlabel{it:7.3b} \item
$r$ satisfies Eqs.~(\meqref{eq:6.18})-(\meqref{eq:6.19});
\mlabel{it:7.3c} \item $T_r$ is an $\mathcal{O}$-operator  of $A$
associated to the bimodule
$(A^*,R_{\searrow}^*,L_{\star}^*,-R_{\vee}^*,
-L_{\prec}^*,-R_{\succ}^*,-L_{\wedge}^*,R_{\star}^*$,
$L_{\nwarrow}^*)$; \mlabel{it:7.3d} \item $T_r$ is an
$\mathcal{O}$-operator of $A_v$ associated to the bimodule
$(A^*,-R_{\swarrow}^*,L_{\succ}^*,R_{\prec}^*, -L_{\nearrow}^*)$.
\mlabel{it:7.3e}
\end{enumerate}
\mlabel{pp:7.3}
\end{prop}

\begin{defn}{\rm \cite{Bai3} Let $(A,\nwarrow,\nearrow,\swarrow,\searrow)$ be a
quadri-algebra and let $r\in A\otimes A$. A set of
equations~(\meqref{eq:6.14}) and (\meqref{eq:6.15}) is called the
{\bf $Q$-equations} in $(A,\nwarrow,\nearrow,\swarrow,\searrow)$.
} \mlabel{de:7.4}
\end{defn}

\begin{remark}{\rm In the sense of Proposition~\mref{pp:7.3} (in terms of $\mathcal O$-operators), \qeq in
a quadri-algebra can be regarded as an analogue of the classical
Yang-Baxter equation in a Lie algebra~\cite{Bai1,K}, which led to
the introduction of \qeq in~\cite{Bai3}. }
\mlabel{rk:7.5}
\end{remark}

The following conclusion is obvious:

\begin{coro}
 Let
$(A,\nwarrow,\nearrow,\swarrow,\searrow)$ be a quadri-algebra and
let $r\in A\otimes A$ be skew-symmetric. Then the
comultiplications $\alpha,\beta,\tilde{\alpha}$ and
$\tilde{\beta}$ defined by
Eqs.~(\meqref{eq:6.1})-(\meqref{eq:6.4}) with
$r_{\nwarrow}=r_{\nearrow} =r_{\swarrow}=r_{\searrow}=r$
respectively make
$(A,\nwarrow,\nearrow,\swarrow,\searrow,\alpha,\beta,
\tilde{\alpha},\tilde{\beta})$ into a quadri-bialgebra if $r$ is a
solution of \qeq. \mlabel{co:7.6}
\end{coro}

\begin{prop}
Let $(A,\nwarrow,\nearrow,\swarrow,\searrow)$ be a quadri-algebra
and let $r$ be a skew-symmetric solution of \qeq. Then the
quadri-algebra structure
$(\nwarrow_{\bullet},\nearrow_{\bullet},\swarrow_{\bullet},\searrow_{\bullet})$
on the Drinfeld $Q$-double $\mathcal{QD}(A)$ in
Theorem~\mref{thm:6.7} can be given as follows (for all $x\in
A,a^*,b^*\in A^*$): \begin{equation} a^*\nwarrow_{\bullet}b^*=
R_{\searrow}^*(T_r(a^*))b^*+ L_{\star}^*(T_r(b^*))a^*,
a^*\nearrow_{\bullet}b^*=
-R_{\vee}^*(T_r(a^*))b^*-L_{\prec}^*(T_r(b^*))a^*, \mlabel{eq:7.2}
\end{equation}\begin{equation}
a^*\swarrow_{\bullet}b^*=
-R_{\succ}^*(T_r(a^*))b^*-L_{\wedge}^*(T_r(b^*))a^*,
a^*\searrow_{\bullet}b^*=
R_{\star}^*(T_r(a^*))b^*+L_{\nwarrow}^*(T_r(b^*))a^*,
\mlabel{eq:7.3}
\end{equation}\begin{equation}
a^*\nwarrow_{\bullet}x=-T_r(L_{\star}^*(x)a^*)+ T_r(a^*)\nwarrow
x+L_{\star}^*(x)a^*,
\mlabel{eq:7.4}
\end{equation}\begin{equation}a^*\nearrow_{\bullet}x=T_r(L_{\prec}^*(x)a^*)+
T_r(a^*)\nearrow x-L_{\prec}^*(x)a^*,
\mlabel{eq:7.5}
\end{equation}\begin{equation}
a^*\swarrow_{\bullet}x=T_r(L_{\wedge}^*(x)a^*)+ T_r(a^*)\swarrow
x-L_{\wedge}^*(x)a^*,
\mlabel{eq:7.6}
\end{equation}\begin{equation}
a^*\searrow_{\bullet}x=-T_r(L_{\nwarrow}^*(x)a^*)+ T_r(a^*)\searrow
x+L_{\nwarrow}^*(x)a^*,
\mlabel{eq:7.7}
\end{equation}\begin{equation}
x\nwarrow_{\bullet}a^*=R_{\searrow}^*(x)a^*+ x\nwarrow
T_r(a^*)-T_r(R_{\searrow}^*(x)a^*),
\mlabel{eq:7.8}
\end{equation}\begin{equation}
x\nearrow_{\bullet}a^*=-R_{\vee}^*(x)a^*+ x\nearrow
T_r(a^*)+T_r(R_{\vee}^*(x)a^*),
\mlabel{eq:7.9}
\end{equation}\begin{equation}
x\swarrow_{\bullet}a^*=-R_{\succ}^*(x)a^*+ x\swarrow
T_r(a^*)+T_r(R_{\succ}^*(x)a^*),
\mlabel{eq:7.10}
\end{equation}\begin{equation}
x\searrow_{\bullet}a^*=R_{\star}^*(x)a^*+ x\searrow
T_r(a^*)-T_r(R_{\star}^*(x)a^*).
\mlabel{eq:7.11}
\end{equation}
\mlabel{pp:7.7}
\end{prop}

\begin{proof}
Let $\{e_1,...,e_n\}$ be a basis of $A$ and let
$\{e_1^*,...,e_n^*\}$ be the dual basis. Suppose that
$$e_i\nwarrow e_j=\sum_{i,j}c_{ij}^ke_k,e_i\nearrow
e_j=\sum_{i,j}d_{ij}^ke_k, e_i\swarrow
e_j=\sum_{i,j}\tilde{c}_{ij}^ke_k, e_i\searrow e_j=
\sum_{i,j}\tilde{d}_{ij}^ke_k,$$ and $r=\sum_{i,j}a_{ij}e_i\otimes
e_j$, where $a_{ij}=-a_{ji}$. Then $T_r(e_i^*)=\sum_ka_{ki}e_k$.
for all $k,l$, we have
\begin{eqnarray*}
e_k^*\nwarrow_*e_l^*&=&\sum_s\langle e_k^*\otimes
e_l^*,\alpha(e_s)\rangle e_s^*\\
&=&\sum_{t,s}[-a_{kt}\tilde{d}_{st}^l+a_{tl}(c_{ts}^k+d_{ts}^k+
\tilde{c}_{ts}^k+\tilde{d}_{ts}^k)]e_s^*\\
&=&R_{\searrow}^*(T_r(e_k^*))e_l^*+L_{\star}^*(T_r(e_l^*))e_k^*.
\end{eqnarray*}
Hence for all $a^*,b^*\in A^*$, we have that
$a^*\nwarrow_{\bullet}b^*= R_{\searrow}^*(T_r(a^*))b^*+
L_{\star}^*(T_r(b^*))a^*$. Similarly $a^*\nearrow_{\bullet}b^*=
-R_{\vee}^*(T_r(a^*))b^*-L_{\prec}^*(T_r(b^*))a^*$. So
Eq.~(\mref{eq:7.2}) holds. Eq.~(\mref{eq:7.3}) is proved in a
similar way. On the other hand, we have
\begin{eqnarray*}
R_{\searrow}^*(e_k^*)e_l&=&\sum_s\langle
e_l,e_s^*\searrow_*e_k^*\rangle e_s\\
&=&\sum_s\langle
e_l,R_{\star}^*(T_r(e_s^*))e_k^*+
L_{\nwarrow}^*(T_r(e_k^*))e_s^*\rangle e_s\\
&=&\sum_s\langle e_s^*,-T_r(L_{\star}^*(e_l)e_k^*)+
T_r(e_k^*)\nwarrow e_l\rangle
e_s\\
&=&-T_r(L_{\star}^*(e_l)e_k^*)+T_r(e_k^*)\nwarrow e_l.
\end{eqnarray*}
Thus, $e_k^*\nwarrow_{\bullet}e_l=R_{\searrow}^*(e_k^*)e_l+
L_{\star}^*(e_l)e_k^*=-T_r(L_{\star}^*(e_l)e_k^*)+
T_r(e_k^*)\nwarrow e_l+L_{\star}^*(e_l)e_k^*$. Therefore
Eq.~(\meqref{eq:7.4}) holds.
Eqs.~(\meqref{eq:7.5})-(\meqref{eq:7.11}) are verified in the same
way.
\end{proof}

\begin{prop} {\rm \cite{Bai3}} Let $(A,\nwarrow,\nearrow,\swarrow,\searrow)$ be a
quadri-algebra and let $r\in A\otimes A$. Suppose that $r$ is
skew-symmetric and nondegenerate. Then $r$ is a solution of
\qeq if and only if the inverse of the isomorphism
$A^*\rightarrow A$ induced by $r$, regarded as a bilinear form
$\omega$ on $A$, satisfies \begin{equation} \omega(x,y\wedge
z)=-\omega(x\swarrow y,z)+\omega(z\succ x,y),\mlabel{eq:7.12}
\end{equation}
\begin{equation}\omega(x,y\vee z)=\omega(x\prec y,z)-\omega(z\nearrow
x,y),\;\;\forall x,y,z\in A.
\mlabel{eq:7.13}
\end{equation}
\mlabel{pp:7.8}
\end{prop}

\begin{defn}{\rm   Let $(A,\nwarrow,\nearrow,\swarrow,\searrow)$ be a
quadri-algebra. A bilinear form $\omega:A\otimes A\rightarrow
\mathbb F$ is called a {\bf $2$-cocycle} on $A$ if it satisfies
Eqs.~(\meqref{eq:7.12}) and (\meqref{eq:7.13}). } \mlabel{de:7.9}
\end{defn}

It is easy to show that if $\omega$ is a $2$-cocycle on a
quadri-algebra $A$, then for all $x,y\in A$,
$\mathfrak{B}(x,y):=\omega(x,y)+\omega(y,x)$ is a $2$-cocycle on
the associated vertical dendriform algebra $A_v$.

\begin{prop} Let $(A,\nwarrow,\nearrow,\swarrow,\searrow,\alpha,\beta,
\tilde{\alpha},\tilde{\beta})$ be a  quadri-bialgebra obtained from
a skew-symmetric solution of \qeq.  Let
$(A,A^*,R_{\searrow}^*,L_{\star}^*,-R_{\vee}^*, -L_{\prec}^*$,
$-R_{\succ}^*$, $-L_{\wedge}^*,R_{\star}^*,
L_{\nwarrow}^*,R_{\searrow_*}^*,L_{\star_*}^*$, $ -R_{\vee_*}^*,
-L_{\prec_*}^*,-R_{\succ_*}^*,-L_{\wedge_*}^*,R_{\star_*}^*,
L_{\nwarrow_*}^*)$ be the corresponding matched pair of
quadri-algebras, where the subscript $\ast$ denotes the
quadri-algebra structure on $A^*$. Set
\begin{equation}
A\ltimes A^*=:A\ltimes_{R_{\searrow}^*,L_{\star}^*,-R_{\vee}^*,
-L_{\prec}^*,-R_{\succ}^*,-L_{\wedge}^*,R_{\star}^*,
L_{\nwarrow}^*}A^*.
\mlabel{eq:7.14}
\end{equation}
\begin{enumerate}
\item
$A\bowtie_{R_{\searrow}^*,L_{\star}^*,-R_{\vee}^*,
-L_{\prec}^*,-R_{\succ}^*,-L_{\wedge}^*,R_{\star}^*,L_{\nwarrow}^*}^{R_{\searrow_*}^*,
L_{\star_*}^*,-R_{\vee_*}^*,-L_{\prec_*}^*,-R_{\succ_*}^*,-L_{\wedge_*}^*,R_{\star_*}^*,
L_{\nwarrow_*}^*}A^*$ is isomorphic to $A\ltimes A^*$
 as quadri-algebras.
\mlabel{it:7.10a} \item The skew-symmetric solutions of \qeq in
$A$ are in one-to-one correspondence with linear maps
$T_r:A^*\rightarrow A$ whose graphs are the Lagrangian
quadri-subalgebras of $A\ltimes A^*$
 with respect to the bilinear form $\frak B_S$ defined by Eq.~(\meqref{eq:3.2}). Here
the graph of a linear map $T:A^*\rightarrow A$ is defined as
$\mgraph(T):=\{(T(a^*),a^*)|a^*\in A^*\}\subset A\ltimes A^*$.
 \mlabel{it:7.10b}
\end{enumerate}
\mlabel{pp:7.10}
\end{prop}

\begin{proof} We denote the quadri-algebra structure on
$A\ltimes A^*$ by
$\nwarrow_{\dagger},\nearrow_{\dagger},\swarrow_{\dagger},
\searrow_{\dagger}$. Define a linear map
$\theta:A\bowtie_{R_{\searrow}^*,L_{\star}^*,-R_{\vee}^*,
-L_{\prec}^*,-R_{\succ}^*,-L_{\wedge}^*,R_{\star}^*,L_{\nwarrow}^*}^{R_{\searrow_*}^*,
L_{\star_*}^*,-R_{\vee_*}^*,-L_{\prec_*}^*,-R_{\succ_*}^*,-L_{\wedge_*}^*,R_{\star_*}^*,
L_{\nwarrow_*}^*}A^*\rightarrow A\ltimes A^*$ by
$$\theta(x,a^*):=(T_r(a^*)+x,a^*),\forall x\in A,a^*\in A^*.$$ It is
straightforward to check that $\theta$ is a homomorphism of
quadri-algebras. Moreover, $\theta$ is bijective. So
Item~(\meqref{it:7.10a}) holds.

Suppose that $r$ is a skew-symmetric solution of \qeq. Then by
Item~(\meqref{it:7.10a}) we know that $\theta(A^*)=\mgraph(T_r)$
and $\theta(A)=A$ are isotropic complementary quadri-subalgebras
of $A\ltimes A^*$ and dual to each others with respect to the
bilinear form $\frak B_S$. 
So
$\mgraph(T_r)$ is a Lagrangian quadri-subalgebra of $A\ltimes
A^*$. Conversely, let $T_r:A^*\rightarrow A$ be a linear map whose
graph $\mgraph(T_r)$ is a Lagrangian quadri-subalgebra of
$A\ltimes A^*$. Since $\mgraph(T_r)$ is Lagrangian, $r$ is
skew-symmetric. Moreover, let $a^*,b^*\in A^*$. Since
$\mgraph(T_r)$ is a quadri-subalgebra of $A\ltimes A^*$, we have
{\small \begin{eqnarray*}
(T_r(a^*),a^*)\nwarrow_{\dagger}(T_r(b^*),b^*)&=&(T_r(a^*)\nwarrow
T_r(b^*),R_{\searrow}^*(T_r(a^*))b^*+L_{\star}^*(T_r(b^*))a^*)\\
\smallskip
&=&(T_r(R_{\searrow}^*(T_r(a^*))b^*+L_{\star}^*(T_r(b^*))a^*),
R_{\searrow}^*(T_r(a^*))b^*+L_{\star}^*(T_r(b^*))a^*).
\end{eqnarray*}}
Therefore, $T_r(a^*)\nwarrow
T_r(b^*)=T_r(R_{\searrow}^*(T_r(a^*))b^*+L_{\star}^*(T_r(b^*))a^*)$.
Similarly we have
$$T_r(a^*)\nearrow T_r(b^*)=
T_r(-R_{\vee}^*(T_r(a^*))b^*-L_{\prec}^*(T_r(b^*))a^*),$$
$$T_r(a^*)\swarrow T_r(b^*)=T_r(-R_{\succ}^*(T_r(a^*))b^*-
L_{\wedge}^*(T_r(b^*))a^*),$$  $$T_r(a^*)\searrow T_r(b^*)=
T_r(R_{\star}^*(T_r(a^*))b^*+L_{\nwarrow}^*(T_r(b^*))a^*).$$ So by Proposition~\mref{pp:7.3}, we conclude that $r$ is a skew-symmetric solution of \qeq.
\end{proof}

\section{Construction of linear operators on double spaces of quadri-algebras}
\mlabel{sec:double}

An interesting (and important) feature of dendriform algebras,
quadri-algebras and other similar algebra structures mentioned in
\cite{EG1} is that they have close relations with various linear
operators in combinatorics~\cite{Ag1,Ag2,Ag3, AL,
Bax,E1,E2,EG1,R1,R2}.

\begin{defn}{\rm   Let $A$ be a vector space with a set
 of bilinear operations $\Omega:=\{\ast_{n}: A\otimes
A\rightarrow A, n=1,...,m\}$. A linear operator $P$ on $A$ is called a {\bf
Rota-Baxter operator of weight $\lambda$} ($\in\mathbb F$) if, for
each $\ast\in\Omega$, we have \begin{equation} P(x)\ast
P(y)=P(P(x)\ast y+x\ast P(y)+\lambda x\ast y),   \forall x,y\in
A.
\mlabel{eq:8.1}
\end{equation}
A linear operator $N$ on $A$ is called a {\bf
Nijenhuis operator} if, for each $\ast\in\Omega$, we have
\begin{equation} N(x)\ast N(y)=N(N(x)\ast y+x\ast N(y)-N(x\ast y)),
\forall x,y\in A.
\mlabel{eq:8.2}
\end{equation}}
\mlabel{de:8.1}
\end{defn}

We have the following relationship between Rota-Baxter operators of weight $\lambda$ and Nijenhuis operators.

\begin{prop} {\rm ~\cite{E1}}  With the notations in Definition 8.1, if
$N:A\rightarrow A$ is a Nijenhuis operator on $A$ satisfying
$N^2=\lambda^2{\rm id}$, then $P:=$${-\lambda{\rm id}-N}\over 2$ is
a Rota-Baxter operator of weight $\lambda$ (on $A$).
\mlabel{pp:8.2}
\end{prop}

By a direct computation, we have
\begin{prop} Let $(A,\nwarrow,\nearrow,\swarrow,\searrow,\alpha,\beta,
\tilde{\alpha},\tilde{\beta})$ be a quadri-bialgebra, which is
induced from a skew-symmetric solution $r$ of \qeq. If, in
addition, $r$ is nondegenerate, then for all $x\in A,a^*\in A^*$,
\begin{equation} N_{\lambda_1,\lambda_2,\lambda_3,\lambda_4}((x,a^*)):=(\lambda_1T_r(a^*)+
\lambda_2x,\lambda_3T_r^{-1}(x)+\lambda_4a^*)
\mlabel{eq:8.3}
\end{equation} is a
Nijenhuis operator on the Drinfeld $Q$-double $\mathcal{QD}(A)$,
where $\lambda_i\in\mathbb F$, $i=1,2,3,4$.
\mlabel{pp:8.3}
\end{prop}


In the following we assume that the coefficient field $\mathbb F$
is the real number field $\mathbb R$ and fix a $\lambda\in\mathbb
R$. With the conditions in Proposition~\mref{pp:8.3}, if we
consider a Nijenhuis operator satisfying
$N_{\lambda_1,\lambda_2,\lambda_3,\lambda_4}^2=\lambda^2{\rm id}$
and apply Proposition~\mref{pp:8.2}, then we can get three
families of Rota-Baxter operators of weight $\lambda$ on
$\mathcal{QD}(A)$:
\begin{eqnarray*}
&{\rm (F1)}&
P_{\lambda,k}^{+}(x,a^*) :=\frac{-\lambda-N_{0,\lambda,k,-\lambda}}{2}(x,a^*)
=\left(
-\lambda x,-\frac{k}{2}T_r^{-1}(x)\right)\\ & & {\rm or\ }
P_{\lambda,k}^{-}(x,a^*)
:=\frac{-\lambda-N_{0,-\lambda,k,\lambda}}{2}(x,a^*)
=\left(
0,-\frac{k}{2}T_r^{-1}(x)-\lambda a^*\right), \quad  k\neq 0;\\
&{\rm (F2)}&
\hat{P}_{\lambda,\hat{k}}^{+}(x,a^*)
:=\frac{-\lambda-N_{\hat{k},\lambda,0,-
\lambda}}{2}(x,a^*)
=\left(-\frac{\hat{k}}{2}T_r(a^*)-\lambda
x,0\right)\\
& &{\rm or\ }
\hat{P}_{\lambda,\hat{k}}^{-}(x,a^*)
:=\frac{-\lambda-N_{\hat{k},-\lambda,0,
\lambda}}{2}(x,a^*)
=\left(-\frac{\hat{k}}{2}T_r(a^*),-\lambda a^*\right), \quad  (\hat{k},\lambda)\neq (0,0); \\
&{\rm (F3)}&
P_{\lambda,k_1,k_2}(x,a^*)
:=\frac{-\lambda-N_{k_1,k_2,{{\lambda^2-k_2^2}\over{k_1}},-k_2}}{2}(x,a^))\\
& &=
\left(-\frac{k_1}{2}T_r(a^*)-\frac{(k_2+\lambda)}{2}x,\frac{(k_2^2-\lambda^2)}{2k_1}T_r^{-1}(x)+\frac{(k_2
-\lambda)}{2}a^*\right), \quad  k_2\neq\pm \lambda,   k_1\neq 0,
\end{eqnarray*}
for all $x\in A, a^*\in A^*$, where $k,\hat{k},k_1,k_2\in\mathbb
R$. Here we exclude the trivial case when $P=-\lambda{\rm id}$.
Furthermore, it is easy to check that these Rota-Baxter operators
are idempotents (i.e., $P^2=P$) if and only if $\lambda=-1$, that
is, they are Rota-Baxter operators of weight $-1$. We record them
explicitly as follows (see the discussion at the end of this
section for their important roles in renormalization in quantum
field theory):
\begin{eqnarray*}
&{\rm (G1)}&
P_{-1,k}^{+}(x,a^*):=\frac{1-N_{0,-1,k,1}}{2}(x,a^*)
=\left(
x,-\frac{k}{2}T_r^{-1}(x)\right)\\ & & {\rm or\ }
P_{-1,k}^{-}((x,a^*)):=\frac{1-N_{0,1,k,-1}}{2}((x,a^*))=(
0,-\frac{k}{2}T_r^{-1}(x)+a^*), \quad  k\neq 0;\\
&{\rm (G2)}&
\hat{P}_{-1,\hat{k}}^{+}(x,a^*)
:=\frac{1-N_{\hat{k},-1,0,
1}}{2}(x,a^*)
=\left(-\frac{\hat{k}}{2}T_r(a^*)+x,0\right)\\
& &{\rm or\ }
\hat{P}_{-1,\hat{k}}^{-}(x,a^*):=\frac{1-N_{\hat{k},1,0,
-1}}{2}(x,a^*)=\left(-\frac{\hat{k}}{2}T_r(a^*),a^*\right),
\\
&{\rm (G3)}&
P_{-1,k_1,k_2}(x,a^*):=\frac{1-N_{k_1,k_2,{{1-k_2^2}\over{k_1}},-k_2}}{2}(x,a^*)\\
& &=
\left(-\frac{k_1}{2}T_r(a^*)-\frac{(k_2-1)}{2}x,\frac{(k_2^2-1)}{2k_1}T_r^{-1}(x)+\frac{(k_2
+1)}{2}a^*\right),  \quad  k_2\neq\pm 1,   k_1\neq 0,
\end{eqnarray*}
for all $x\in A, a^*\in A^*$, where $k,\hat{k},k_1,k_2\in\mathbb R$.

On the other hand, the requirement that $r$ is nondegenerate can be
dropped if the $\lambda_3$ appearing in the definition of
$N_{\lambda_1,\lambda_2,\lambda_3,\lambda_4}$ equals to zero. More
precisely, we have the following immediate conclusion.

\begin{prop}
Let $(A,\nwarrow,\nearrow,\swarrow,\searrow,\alpha,\beta,
\tilde{\alpha},\tilde{\beta})$ be a quadri-bialgebra, which is
induced from a skew-symmetric solution $r$ of \qeq. Then for all
$x\in A,a^*\in A^*$,
\begin{equation} N_{\lambda_1,\lambda_2,\lambda_3}((x,a^*)):=(\lambda_1T_r(a^*)+
\lambda_2x,\lambda_3a^*)\end{equation} is a Nijenhuis operator on
$\mathcal{QD}(A)$, where $\lambda_i\in\mathbb F$, $i=1,2,3$.
\mlabel{pp:8.4}
\end{prop}

\begin{remark}{\rm   Similarly, we consider the case when
$N_{\lambda_1,\lambda_2,\lambda_3}=\lambda^2{\rm id}$ for
$\lambda\in\mathbb R$ and then apply Proposition~\mref{pp:8.2} to get certain
families of Rota-Baxter operators (of weight $\lambda$) on
$\mathcal{QD}(A)$. In fact, one can show that the Rota-Baxter
operators (of weight $\lambda$) are given by {\rm (F2)} and the
Rota-Baxter operators which are idempotents are given by {\rm (G2)},
where $x\in A,a^*\in A^*$.}
\mlabel{rk:8.5}
\end{remark}

\begin{prop}  Let $(A,\nwarrow,\nearrow,\swarrow,\searrow)$ be a
quadri-algebra. Then
\begin{equation}\mlabel{eq:8.5}N_{\lambda_1,\lambda_2,\lambda_3,\lambda_4}((x,y)):=(\lambda_1y+\lambda_2x,
\lambda_3x+\lambda_4y),\;\;\forall x,y\in A\end{equation} is a
Nijenhuis operator on $A_v\ltimes_{L_{\nearrow},R_{\nwarrow},
L_{\searrow},R_{\swarrow}}A$,
 where $x,y\in A$ and $\lambda_i\in\mathbb F$, $i=1,2,3,4$.
 Moreover, if $(A_v,A_v,L_{\nearrow},R_{\nwarrow},
L_{\searrow},R_{\swarrow},L_{\nearrow},R_{\nwarrow},
L_{\searrow},R_{\swarrow})$  is a matched pair of dendriform
algebras, then the linear operator defined by Eq.~(\mref{eq:8.5})
is also a Nijenhuis operator on
$A_v\bowtie_{L_{\nearrow},R_{\nwarrow},
L_{\searrow},R_{\swarrow}}^{L_{\nearrow},R_{\nwarrow},
L_{\searrow},R_{\swarrow}}A_v$.  On the other hand,
\begin{equation}
\theta((x,y)):=(y+x,x),\;\;\forall x,y\in A\end{equation} is an
isomorphism of dendriform algebras from
$A_v\bowtie_{L_{\nearrow},R_{\nwarrow},
L_{\searrow},R_{\swarrow}}^{L_{\nearrow},R_{\nwarrow},
L_{\searrow},R_{\swarrow}}A_v$  to
$A_v\ltimes_{L_{\nearrow},R_{\nwarrow},
L_{\searrow},R_{\swarrow}}A$. \mlabel{pp:8.6}
\end{prop}

\begin{proof} It is straightforward.\end{proof}

\begin{remark}{\rm  Similarly, we can also consider the case that
$N_{\lambda_1,\lambda_2,\lambda_3,\lambda_4}=\lambda^2{\rm id}$
and then apply Proposition~\ref{pp:8.2} to get certain families of
Rota-Baxter operators of weight $\lambda$ on the double spaces
given in Proposition~\mref{pp:8.6}.} \mlabel{rk:8.7}
\end{remark}

\begin{remark}{\rm  One can use these Nijenhuis operators and
Rota-Baxter operators (of weight $\lambda$) to construct
$NS$-algebra and tridendriform algebra structures (see \cite{EG1}
and the references therein). Moreover, it is easy to show that
these Nijenhuis operators also satisfy Eq.~(\meqref{eq:8.2}) on
their associated associative algebras, that is, they are the
so-called {\bf associative Nijenhuis tensors} of the associated
associative algebras in the sense of \cite{CGM}, where this notion
was introduced in the study of Wigner problem in quantum physics.}
\mlabel{rk:8.8}
\end{remark}

\begin{remark}{\rm Obviously, if $P$ is a Rota-Baxter operator of weight
$\lambda$ on a quadri-algebra, then it is a Rota-Baxter operator
of weight $\lambda$ on the associated associative algebra.
Furthermore, Rota-Baxter operators on associative algebras which
are idempotents play a key role in the {\bf algebraic Birkhoff
decomposition} in pQFT~\cite{CK,EGK1,EGK2}. See~\cite{EG2} also
for a nice survey of this topic. As discussed above, we have
constructed certain families of Rota-Baxter operators on the
double spaces of quadri-algebras.} \mlabel{rk:8.9}
\end{remark}

\bigskip

\noindent {\bf Acknowledgements.} The authors are grateful to
Professor J.-L. Loday for his inspiring suggestions and to L. Guo
for helpful discussions. This work is supported by NSFC
(11931009).  C. Bai is also supported by the Fundamental Research
Funds for the Central Universities and Nankai ZhiDe Foundation.

\end{document}